\documentclass[12pt]{article} 
\usepackage{amssymb}
\usepackage{amsmath,bm}
\usepackage{array}
\usepackage{graphics}
\usepackage{color}
\usepackage{color}
\usepackage[color,matrix,arrow]{xy}
\usepackage{xcolor}
\usepackage{xypic}
\usepackage{tikz}
\usepackage[utf8]{inputenc}                    
\usepackage{microtype}                         
\usepackage{mathtools, amsthm, amssymb, eucal} 
\usepackage[normalem]{ulem}                    
\usepackage{mdframed}
\usepackage{verbatim}
\usepackage{booktabs} 
\usepackage{mathrsfs}
\usepackage{titling}
\usepackage{xypic}
\usepackage[paper=letterpaper, margin=1in, headsep=20pt]{geometry}
\usepackage{hyperref}
\usepackage{enumerate}
\usepackage{multicol}
    \theoremstyle{plain}
    \newtheorem{thm}{Theorem}[section]
    \newtheorem{lem}[thm]{Lemma}
    \newtheorem{prop}[thm]{Proposition}

    \theoremstyle{definition}
    \newtheorem{defn}{Definition}[section]
    
    \newtheorem{exmp}{Example}[section]

    \theoremstyle{remark}
    \newtheorem{rem}{Remark}[section]

\definecolor{nik}{RGB}{200,0,0}

\definecolor{ray}{RGB}{0,0,200}
\newcommand{\raysays}[1]{\marginpar{\color{ray}#1\color{black}}}

\begin{document}

\centerline{\textbf{\Large Enumerating families of clusters in type $\tilde{\mathbb{A}}$}}
\medskip
\centerline{Ray Maresca and Nikolas Proskura}

\begin{abstract}
In this paper, we will place clusters in type $\tilde{\mathbb{A}}$ (equivalently triangluations of an annulus) into infinite families parametrized by winding numbers of certain arcs in the corresponding triangulation. We will count how many such families there are and provide an algebraic description of these families in terms of the corresponding cluster category.
\end{abstract}

\section{Introduction}
Since it's introduction by Fomin and Zelevinsky in the early 2000's, cluster theory has become a vibrant area of research in modern algebra. A large class of cluster algebras are those that arise from a surface. As shown by Fomin, Shaprio, and Thurston in \cite{clustersareinbijectionwithtriangulations}, cluster algebras of type $\tilde{\mathbb{A}}$ arise from annuli with a distinct set of marked points on the boundary components. The `clusters' correspond to triangulations of the annulus from which the cluster algebra arises.

It is well-known that there are infinitely many such triangulations in type $\tilde{\mathbb{A}}$. In \cite{IgusaMaresca}, Igusa and the first author placed the triangulations, hence clusters, into finitely many infinite families by considering triangulations up to $2\pi$ Dehn twists of the inner boundary component of the annulus. Moreover, Igusa and the first author counted how many such families there are for one particular orientation; namely, when the annulus has one marked point on the outer boundary component and $n$ marked points on the inner component. In Section \ref{sec: counting method}, we prove our first main result, Theorem \ref{finalresult}, in which we count the number of families of triangulations, hence clusters, when the annulus has $n$ marked points on the outer circle and $m$ marked points on the inner circle. In Section \ref{sec: algebraic description of families}, we will prove our second main result, Theorem \ref{thm: families in cluster category}, which provides an algebraic description of the families in terms of the corresponding cluster category. Finally in Section \ref{sec: examples} we will provide a table containing the number of families of clusters for certain annuli, and give an example of all the families for one particular orientation.

\section{Preliminaries}

Let $\Bbbk = \overline{\Bbbk}$ be an algebraically closed field. Let $Q=(Q_0,Q_1,s,t)$ denote a quiver with vertex set $Q_0$, arrow set $Q_1$ and functions $s,t:Q_1 \rightarrow Q_0$ that assign to each arrow a starting and terminal point respectively. A \textbf{path of length} $m \geq1$ in $Q$ is a finite sequence of arrows $\alpha_1\alpha_2\dots\alpha_m$ where $t(\alpha_j) = s(\alpha_{j+1})$ for all $1 \leq j \leq m-1$. We denote the path algebra of $Q$ by $\Bbbk Q$. A \textbf{monomial relation} in $Q$ is given by a path of the form $\alpha_1\alpha_2\dots\alpha_m$ where $m\geq 2$. A two sided ideal of $\Bbbk Q$ is called \textbf{admissible} if $R^m_Q \subset I \subset R^2_Q$ for some $m \geq 2$ where $R_Q$ is the arrow ideal in $\Bbbk Q$.  The path algebra $\Bbbk Q$ is a hereditary algebra (submodules of projective modules are projective) when $Q$ is acyclic. Finitely generated $\Bbbk Q$-modules are finite dimensional representations of $Q$. These form the category $mod$-$\Bbbk Q$. An object $V$ in this category is \textbf{partial tilting} if Ext$_{\Bbbk Q}^1 (V,V) = 0$ and it is $\textbf{tilting}$ if in addition the number of indecomposable summands of $V$ equals the number of vertices in the quiver.

\begin{defn}
 A \textbf{cluster algebra} $\mathscr{A}_Q$ is a subalgebra $\mathscr{A}\subset \mathbb{Q}(x_1,x_2,\dots, x_n)$ generated by the cluster variables given by mutating a seed $(Q,x_{*})$ where
\begin{itemize}
\item $Q$ is a quiver with no loops or two cycles. 
\item $x_{*} = (x_1,x_2,\dots,x_n)$ is a transcendence basis for $\mathbb{Q}(x_1,x_2,\dots, x_n) = \{{f(x)\over g(x)} : f,g \in \mathbb{Q}(x_1,x_2,\dots, x_n) \}$. 
\item A \textbf{cluster} is a collection of $n$ cluster variables attained by mutations.
\end{itemize} 
\end{defn}

In order for this definition to be complete, we must define mutation which consists of two parts, mutation of the quiver and the seed. We begin by defining mutation of a quiver $Q$ with a running example. For the following definition, let $Q'$ be the quiver 

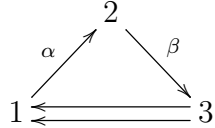
\begin{center}
\begin{tabular}{c}
\xymatrix{                             & 2 \ar[dr]^{\beta}   &                             \\
				1 \ar[ur]^{\alpha}     &                  & 3 \ar@<.5ex>[ll] \ar@<-.5ex>[ll]       }
\end{tabular}
\end{center}

\begin{defn} \label{def: quiver mutation}
We define the \textbf{mutation of $Q$ at vertex $k$}, denoted by $\mu_kQ$, as follows.

\begin{enumerate}
\item We first compose any length two paths through $k$ by introducing a new arrow. In the running example, let $k=2$.

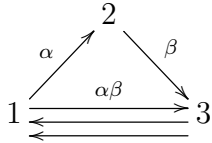
\begin{center}
\begin{tabular}{c}
\xymatrix{                             & 2 \ar[dr]^{\beta}   &                             \\
				                                            1 \ar[ur]^{\alpha} \ar@<.5ex>[rr]^{\alpha\beta}    &                  & 3 \ar@<.5ex>[ll] \ar@<1.5ex>[ll]        }
\end{tabular}
\end{center}
\item Reverse all arrows at $k$:
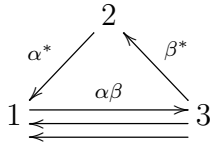
\begin{center}
\begin{tabular}{c}
\xymatrix{                             & 2 \ar[dl]_{\alpha^*}   &                             \\
				                                            1  \ar@<.5ex>[rr]^{\alpha\beta}    &                  & 3 \ar@<.5ex>[ll] \ar@<1.5ex>[ll] \ar[ul]_{\beta^*}       }
\end{tabular}
\end{center}

\item Eliminate all two cycles:

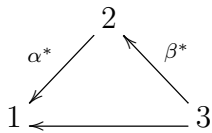
\begin{center}
\begin{tabular}{c}
\xymatrix{                             & 2 \ar[dl]_{\alpha^*}   &                             \\
				                                            1    &                  & 3 \ar@<.5ex>[ll] \ar[ul]_{\beta^*}       }
\end{tabular}
\end{center}

\end{enumerate}

After all three steps, the resulting quiver is $\mu_kQ$, or in this example $\mu_2Q'$. 
\end{defn}

\begin{defn} \label{defn: exchange matrix}
Let $Q$ be a quiver. The \textbf{exchange matrix} of $Q$ is the matrix $B = [b_{ij}]$ where $b_{ij} =$ (the number of arrows $i\rightarrow j) \, -$ (the number of arrows $j \rightarrow i$).
\end{defn} 

\begin{exmp}\label{exmp: exchange matrix}
For the running example in Definition \ref{def: quiver mutation}, the exchange matrix is $$\left[ \vcenter{\xymatrixrowsep{10pt}\xymatrixcolsep{10pt}\xymatrix{0&1&-1\\-1&0&2\\1&-2&0}} \right]$$
\end{exmp}

\begin{defn}\label{defn: mutation of seed}
We define the \textbf{mutation of $x_*$ at vertex $k$}, denoted by $\mu_k(x_*)$, as $\mu_k(x_*) = (x_1, x_2, \dots , x'_k, x_{k+1}, \dots, x_n)$ where $$x'_k = {\displaystyle \prod_{i\rightarrow k} x_i^{b_{ik}} + \displaystyle \prod_{k \rightarrow j} x_j^{b_{kj}} \over x_k}.$$
The notation $\displaystyle \prod_{i\rightarrow k}$ means we take the product over all vertices $i$ such that there is an arrow from $i$ to $k$. We adopt the standard convention that the empty product equals 1.
\end{defn}

Clusters have been categorified, by which we mean that they have been realized as objects in a category. Let $\mathscr{D} := \mathscr{D}^b(mod\text{-}\Bbbk Q)$ denote the bounded derived category of $mod$-$\Bbbk Q$. It is known that $\mathscr{D}$ is a triangulated category with shift functor $\Sigma:\mathscr{D} \rightarrow \mathscr{D}$ and almost split triangles induced by almost split sequences in $mod$-$\Bbbk Q$. Let $\tau: \mathscr{D} \rightarrow \mathscr{D}$ denote the equivalence which induces the Auslander--Reiten (AR) translation so that $\tau C = A$ if we have a triangle of the form $A\rightarrow B \rightarrow C \rightarrow \Sigma A$. The \textbf{cluster category}, defined by Buan, Marsh, Reiten, Reineke, and Todorov in \cite{BMRRT} and denoted by $\mathcal{C}_Q$, is the orbit category $\mathscr{D}/F$ where $F$ is the auto-equivalence given by $F = \Sigma \tau^{-1}:\mathscr{D} \rightarrow \mathscr{D}$. An object $T \in$  Ob$(\mathcal{C}_Q)$ is called a \textbf{cluster} if it is a tilting object in $\mathcal{C}_Q$; that is, Ext$_{\mathcal{C}_Q}^1 (T,T) = 0$ and the number of indecomposable summands of $T$ is $|Q_0|$. For more on tilting theory and general representations of quivers see \cite{BlueBook} and \cite{SchifflerQuiverReps}. For more on cluster categories see \cite{reiten2010cluster}.

 Our focus in this paper is on clusters of type $\tilde{\mathbb{A}}$. To make a graph of type $\tilde{\mathbb{A}}_{p-1}$ a quiver, we define an \textbf{orientation vector} $\bm{\varepsilon} = (\varepsilon_1, \dots , \varepsilon_{p}) \in \{-,+\}^{n+1}$. Then \textbf{the quiver of type $\tilde{\mathbb{A}}_{p-1}$} with this orientation, denoted by $Q^{\bm{\varepsilon}} = (Q_0^{\bm{\varepsilon}},Q_1^{\bm{\varepsilon}},s,t)$, is the one such that $Q_0^{\bm{\varepsilon}} = \{1,2,\dots,p-1,p\}$ and for each $\alpha_i \in Q_1^{\bm{\varepsilon}}$ with $1\leq i \leq p$, we define $\alpha_i \in Q^{\varepsilon}_1$ as

\vspace{-.7cm}

\begin{center}
\begin{multicols}{2}

 \begin{displaymath}
   \alpha_i = \left\{
     \begin{array}{lr}
       i \rightarrow i+1 & : \varepsilon_i = +\\
       i \leftarrow i+1 & :  \varepsilon_i = -
     \end{array}
   \right.
\end{displaymath}

\columnbreak

 \begin{displaymath}
   \alpha_p = \left\{
     \begin{array}{lr}
       p \rightarrow 1 & : \varepsilon_0 = +\\
       p \leftarrow 1 & :  \varepsilon_0 = -
     \end{array}
   \right.
\end{displaymath}

\end{multicols}
\end{center}

Note that any quiver of type $\tilde{\mathbb{A}}_{p-1}$ is given in this way. Moreover, note that when considering our vertex set modulo $|Q_0| = p$, the convention for $\alpha_p$ is consistent with that of $\alpha_i$. So long as $\varepsilon_i \neq \varepsilon_j$ for some $i$ and $j$, these quivers are both hereditary and tame. We say a quiver $Q$ of type $\tilde{\mathbb{A}}_{p-1}$ has \textbf{straight orientation} if $\varepsilon_i = +$ for all $i \in \{1, 2, \dots , p-1\}$ and $\varepsilon_0 = -$. Moreover, we assume the vertices are labeled from left to right in natural numerical order:
\[
\xymatrix{
\tilde{\mathbb{A}}_{p-1}:&1 \ar[r]\ar@/^1.5pc/[rrrr]&2\ar[r]&\quad \cdots\quad  \ar[r]& p-1\ar[r] & p
	}
\]
By \textbf{tame}, we mean there are infinitely many indecomposable $\Bbbk Q$-modules and for all $n\in\mathbb{N}$, all but finitely many isomorphism classes of $n$-dimensional indecomposables occur in a finite number of one-parameter families. It is known that the module category, hence the Auslander--Reiten quiver, denoted by $\Gamma_{\Bbbk Q}$, of a tame hereditary algebra can be partitioned into three sections: the preprojective, $\mathcal{P}$, regular, $\mathcal{R}$, and preinjective, $\mathcal{I}$ components. For $\mathbb{\tilde{A}}_{p-1}$ quivers, the regular component consists of the left, right and homogeneous tubes. We will denote by $P(i), I(i),$ and $S(i)$ the indecomposable projective, injective and simple representation at vertex $i$ respectively. Moreover, the path algebras of quivers of type $\tilde{\mathbb{A}}$ have an additional structure that simplifies the aforementioned tripartite classification of $\Gamma_{\Bbbk Q}$.

\begin{defn}
For an admissible ideal $I$, the algebra $B = \Bbbk Q/ I$ is a \textbf{string algebra} if
\begin{enumerate}
\item At each vertex of $Q$, there are at most two incoming arrows and at most two outgoing arrows.
\item For each arrow $\beta$ there is at most one arrow $\alpha$ and at most one arrow $\gamma$ such that $\alpha\beta \notin I$ and $\beta\gamma \notin I$.  \\

If moreover, we have the following two conditions, the string algebra $B$ is called \textbf{gentle}.
\item The ideal $I$ is generated by a set of monomials of length two.
\item For every arrow $\alpha$, there is at most one arrow $\beta$ and one arrow $\gamma$ such that $0 \neq \alpha\beta \in I$ and $0 \neq \gamma\alpha \notin I$.
\end{enumerate}
\end{defn}

It is well known that the indecomposable modules over string algebras are either string or band modules \cite{BR}. For $\Bbbk Q/I$ a sting algebra, to define string modules, we first define for $\alpha\in Q_1$ a \textbf{formal inverse} $\alpha^{-1}$, such that $s(\alpha^{-1}) = t(\alpha)$ and $t(\alpha^{-1}) = s(\alpha)$. Let $Q_1^{-1}$ denote the set of formal inverses of arrows in $Q_1$. We call arrows in $Q_1$ \textbf{direct} arrows and those in $Q_1^{-1}$ \textbf{inverse} arrows. We define a \textbf{walk} as a sequence $\omega = \omega_0\dots \omega_r$ such that for all $i\in\{0,1,\dots,r-1\}$, we have $t(\omega_i) = s(\omega_{i+1})$ where $\omega_i \in Q_1 \cup Q_1^{-1}$. A \textbf{string} is a walk $\omega$ with no sub-walk $\alpha\alpha^{-1}$ or $\alpha^{-1}\alpha$. A \textbf{band} $\beta = \beta_1\dots \beta_n$ is a cyclic string, that is, $t(\beta_n) = s(\beta_1)$. We define the \textbf{start or beginning} of a string $S = \omega_0\dots \omega_r$, denoted by $s(S)$, as $s(\omega_0)$. Similarly, we define the \textbf{end} of a string $S$, denoted by $t(S)$, as $t(\omega_r)$. For quivers of type $\mathbb{\tilde{A}}_{p-1}$ the band modules lie in the homogeneous tubes and we can classify in which component of the Auslander--Reiten quiver the string modules reside by their shape, as we will soon see. For quivers of type $\tilde{\mathbb{A}}$, we take the convention that all named strings move in the counter-clockwise direction around the quiver $Q$. We denote by $ij_k$ the string module of length $k$ associated to the walk $\omega_1\cdots \omega_{k-1}$ where $s(\omega_1)=i$ and $t(\omega_{k-1})=j$. For $k=1$, $jj_1$ denotes the simple module at vertex $j$, associated to the walk $e_j$ where $e_j$ is the lazy path at vertex $j$. When drawing the graph associated to strings, we take the convention that the head of each arrow is at the bottom. An example of a graph of a string and its associated representation can be seen in Example \ref{exam: example of string module}.

String modules are useful, since their combinatorial structure can be used to describe both the AR translate $\tau$ \cite{BR}, and morphisms/extensions between string modules. To provide the description of the AR translate, we need some definitions. The following lemma follows from the definition of a string algebra.

\begin{lem}
Let $S$ be a string of positive length and let $\varepsilon \in \{ Q, Q^{-1} \}$. There is at most one way to add an arrow preceding $s(S)$ whose orientation agrees with $\varepsilon$, such that the resulting walk is still a string. Similarly there is at most one way to add such an arrow following $t(S)$.
\end{lem}

If a string $S$ can be extended at its end by a direct arrow, then we can add a direct arrow at $t(S)$, followed by adding as many inverse arrows as possible. This operation is called \textbf{adding a cohook} at $t(S)$. Similarly, if $S$ can be extended by an inverse arrow at $s(S)$, then we can add an inverse arrow at $s(S)$, followed by as many direct arrows as possible. This operation is called \textbf{adding a cohook} at $s(S)$. The inverse operation is called \textbf{deleting a cohook}. To do this, we find the last direct arrow in $S$ and remove it along with all the subsequent inverse arrows. This operation is called \textbf{deleting a cohook} at the end of $S$. Similarly, we can find the first inverse arrow in $S$ and remove it along with all the preceding direct arrows. This operation is called \textbf{deleting a cohook} at the start of $S$. Note that deleting a cohook at the end (start) of $S$ may not be defined if $S$ does not any direct (inverse) arrows.

There are dual notions to this, namely \textbf{adding a hook} and \textbf{deleting a hook} respectively. If a string $S$ can be extended at its start by a direct arrow, we add the direct arrow at $s(S)$ followed by as many inverse arrows as possible. This operation is called \textbf{adding a hook} at the start of $S$. If a string $S$ can be extended at its end by adding an inverse arrow, we add this inverse arrow to $S$ along with as many direct arrows as possible. This operation is called \textbf{adding a hook} at the end of $S$. To \textbf{delete a hook} from the end of $S$, we find the last inverse arrow in $S$, and remove it along with all its subsequent direct arrows. Analogously, to \textbf{delete a hook} from the start of $S$, we find the first direct arrow of $S$ and remove it along with all preceding inverse arrows. Again, note that deleting a hook at the end (start) of $S$ may not be defined if $S$ does not have any inverse (direct) arrows. It is known that in type $\tilde{\mathbb{A}}$, all irreducible morphisms in the preprojective (preinjective) component are given by adding hooks (deleting cohooks). With these combinatorial notions, the following theorem was proven in \cite{BR}.

\begin{thm}\label{thm: tau of string algebra}
Let $B = \Bbbk Q/I$ be a string algebra, and let $S$ be a string. At either end of $S$, if it is possible to add a cohook, do it. Then, at the ends at which it was not possible to add a cohook, delete a hook. The result is $\tau_BS$. Inversely, at either end of $S$, if it is possible to delete a cohook, do it. Then, at the ends at which it was not possible to delete a cohook, add a hook. The result is $\tau^{-1}_BS$. \hfill \qed
\end{thm}

For more on representation theory of tame algebras and definitions of these components see \cite{bluebook2}. For more on string algebras see \cite{StringAlgebraInfo} and \cite{CBstrings}. \\

\begin{exmp}   \label{exam: example of string module}

Let $Q$ be the following quiver.

\[
\xymatrix{
1 \ar_{\alpha_1}[r] \ar@/^2pc/^{\alpha_4}[rrr] & 2 \ar_{\alpha_2}[r] & 3 \ar_{\alpha_3}[r] & 4 
	}
\]

Consider the walk $\alpha_3\alpha_4^{-1}\alpha_1\alpha_2\alpha_3$. Then the string module associated to this walk is $34_6$, and we draw its graph as follows.

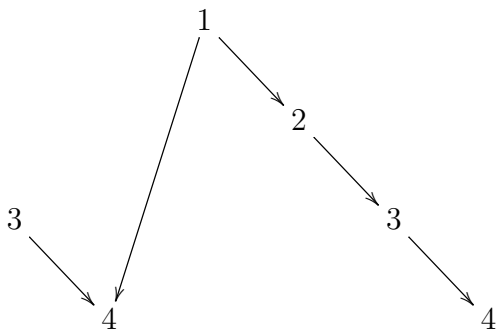
\[
\xymatrix{
 & & & 1 \ar[dddl] \ar[dr] & & & \\ & &  & & 2 \ar[dr] & & \\ & 3 \ar[dr] & & & & 3\ar[dr] & \\ & & 4 & & & & 4
	}
\]

The representation associated to this string is the following.

\[
\xymatrix{
\Bbbk \ar_{1}[r] \ar@/^2pc/^{\begin{bmatrix} 0 \\ 1\end{bmatrix}}[rrr] & \Bbbk \ar_{\begin{bmatrix} 1 \\ 0\end{bmatrix}}[r] & \Bbbk^2 \ar_{\begin{bmatrix} 1 & 0 \\ 0 & 1\end{bmatrix}}[r] & \Bbbk^2 
	}
\]

\end{exmp}

Let $w = w_0\cdots w_r$ be an indecomposable $\Bbbk \mathbb{\tilde{A}}_{p-1}$-module. It is well known that $w$ is \textbf{preprojective} if and only if there are arrows $\alpha,\beta\in Q_1$ such that $t(\alpha) = s(w_0)$ and $t(\beta) = t(w_r)$. Similarly, $w$ is \textbf{preinjective} if and only if there are arrows $\alpha,\beta\in Q_1$ such that $s(\alpha) = s(w_0)$ and $s(\beta) = t(w_r)$, $w$ is \textbf{left regular} if and only if there are arrows $\alpha,\beta\in Q_1$ such that $t(\alpha) = s(w_0)$ and $s(\beta) = t(w_r)$, $w$ is \textbf{right regular} if and only if there are arrows $\alpha,\beta\in Q_1$ such that $s(\alpha) = s(w_0)$ and $t(\beta) = t(w_r)$ and finally $w$ is \textbf{homogeneous} if and only if $w$ is a band. For instance, the $34_6$ string is a preprojective module over the path algebra of the quiver in Example \ref{exam: example of string module}.

\subsection{Cluster combinatorics in type $\tilde{\mathbb{A}}$}\label{sec: cluster combinatorics}

\begin{defn}
An \textbf{annulus} is any compact space homeomorphic to $\boldsymbol{S}^1 \times [0,1]$, where $\boldsymbol{S}^n$ is an n-sphere.
\end{defn}

\begin{rem} 
For some fixed $a,b\in \mathbb{R}$, the set $\{x\in \mathbb{R}^2 | a\leq \lVert x \lVert \leq b \}$ is an annulus.
\end{rem}

It is known that clusters in type $\tilde{\mathbb{A}}$ are in bijection with triangulations of an annulus associated to the quiver $Q^{\bm{\varepsilon}}$ \cite{clustersareinbijectionwithtriangulations}. In this section, we will explicitly provide this bijection. To a quiver $Q^{\bm{\varepsilon}}$ of type $\tilde{\mathbb{A}}_{p-1}$, we associate an annulus $A_{Q^{\bm{\varepsilon}}}$ as follows. If $\varepsilon_i = +(-)$, then there is a marked point on the outer (inner) circle of the annulus corresponding to the vertex $i$ in $Q_0^{\bm{\varepsilon}}$. We moreover write the marked points in clockwise order respecting the natural numerical order of the vertices in the quiver.

The following theorem is well-known. For instance the statement is used in \cite{Master'sThesisClustersandTriangulations}; however, the map given in the proof is not explicitly used.

\begin{thm} \label{map}
    The universal cover of an annulus $A$ is $\mathbb{R} \times [0,1]$.
\end{thm}
\begin{proof}
    The map $ p: \mathbb{R} \times [0,1] \rightarrow A$, given by $(x,y)\mapsto (e^{2\pi ix},y)$, is the covering map. 
\end{proof}

Let $A_{n,m}$ be an annulus with $n$ vertices on the outside boundary component, and $m$ vertices on the inside boundary component. Up to homeomorphism we can assume that the points on the outer and inner boundary components are equally spaced. We can then denote the set of equally spaced outside (respectively inside) vertices by $\mathcal{A} = \{ a_1, ..., a_n\}$ (respectively $\mathcal{B} = \{ b_1, ... , b_m \}$) such that the \textbf{successor} of a vertex $a_{i\bmod n}$, denoted by $Suc(a_{i\bmod m})$, is $a_{i+1\bmod n}$, the vertex immediately clockwise of $a_{i\bmod n}$. The points in $\mathcal{B}$ satisfy an analogous relationship.

\begin{defn}\label{defn: arcs on annulus}
An \textbf{arc} on the annulus $A_{n,m}$ is an isotopy class of simple curves $a(x,y)[\lambda]$ with $x,y \in \mathcal{A} \cup \mathcal{B}$ and $\lambda \in \mathbb{Z}$ satisfying the following:

\begin{itemize}

\item any curve $\gamma\in a(x,y)[\lambda]$ has endpoints $x$ and $y$.
\item any curve $\gamma\in a(x,y)[\lambda]$ travels clockwise through the interior of the annulus from $x$ to $y$.
\item  If $\gamma$ begins and ends on the same boundary component, $\gamma$ is called an \textbf{exterior or boundary arc} and the integer $\lambda$ is the winding number of $\gamma$ about the inner boundary component. If $\gamma$ connects the two boundary components, $\gamma$ is called a \textbf{bridging arc} and $\lambda$ is the clockwise winding number of $\gamma$ about the inner boundary circle of $A_{n,m}$ when traversing $\gamma$ beginning from the outer boundary component. If an arc from $x$ to $y$ has counter-clockwise winding number $k$, we write $a(x,y)[-k]$.

\item We call an arc \textbf{small} if its winding number is 0.

\item A \textbf{triangulation} of the annulus $A_{n,m}$ is a maximal collection of $n+m=|Q_0^{\bm{\varepsilon}}| = p$ noncrossing arcs that are not homotopic to a piece of a boundary component.

\item We call a triangulation \textbf{small} if all the bridging arcs are small.

\end{itemize}

\end{defn}

\begin{rem}
    For some notational simplicity, unless the winding number of the arc is important, we will often disregard it.
\end{rem}

\begin{rem}
    Note that in any triangulation of $A_{n,m}$, to avoid crossings, the winding number of each exterior arc is at most 1, and there can be only one exterior arc on each boundary component with winding number 1. Moreover, to avoid crossings, the winding number of bridging arcs can differ by at most 1. 
\end{rem}

The relationship between the annulus and its universal cover has been described in \cite{Master'sThesisClustersandTriangulations}. We will now provide a similar construction with slightly different notations and conventions which will be more convenient for this paper. For instance, in contrast to \cite{Master'sThesisClustersandTriangulations}, for us, the set of marked points will be set of cosests  $\{ \frac{i}{n} +\mathbb{Z}\}\times \{0\}$ for $1\leq i \leq n$ on the lower boundary $\mathbb{R} \times \{0\}$ and $\{ \frac{j}{m}+\mathbb{Z} \}\times \{1\}$ for $1\leq j \leq m$ on the upper boundary $\mathbb{R} \times \{1\}$ due to our map $p$ defined in the proof of Theorem \ref{map}. If we restrict $p$ to $\mathbb{R}/\mathbb{Z} \times [0,1]$ we can in fact define $p^{-1}$ on $\mathcal{A}$ by $a'_i:=p^{-1}(a_i)=\frac{i}{n}$ for $1\leq i \leq n$ (respectively $b'_j:=p^{-1}(b_j)=\frac{j}{m}$ for $b_j\in \mathcal{B}$ and $1\leq j \leq m$). It follows that the preimage of $a_i \in \mathcal{A}$ is $\mathcal{A}_i' \vcentcolon = p^{-1}(a_i) = \{(a'_i + \mathbb{Z},1)\}$, respectively $ \mathcal{B}_j' \vcentcolon = p^{-1}(b_j) = \{(b'_j + \mathbb{Z},0) \}$ for $1\leq i \leq n$, and $1\leq j \leq m$. Let $\mathcal{A}'  = \cup_{i=1}^{n}\mathcal{A}'_i$ and define $\mathcal{B}'$ analogously. We can then consider the x-coordinates of both sets, denoted by $(\mathcal{A}')_x \, ((\mathcal{B}')_x \text{ respectively})$, and apply the $2\pi \epsilon$ Dehn twist obtained by Lemma $\ref{Shifting Lemma}$ to get $\mathcal{A'}_x \cap \mathcal{B'}_x = \emptyset$. Geometrically speaking, this shows that no $a'_i$ has the same x-coordinate as some $b'_j$. If we now consider a region of area 1 within $\mathbb{R} \times [0,1]$, by the nature of our homeomorphism every $a_i$ and $b_j$ is represented by at least one $a'_i$ and $b'_j$ respectively in our region of area 1. 

\begin{rem}
    In this region of area 1 constructed above, there is exactly one $a_k$ and $b_l$ that are represented twice, by $a'_k$, and $a'_k+(1,0)$ as well as $b'_l$ and $b'_l+(1,0)$.
\end{rem}

Thus, for any region $P$ of unit area within the universal cover, we have a relation $<$ given by the natural ordering of the $x$-coordinate of the points within $P$. There are two different types of chords that can be drawn in the universal cover of $A_{n,m}$. The  preimage of a bridging arc $a(x,y)$ under the covering map $p$ is the isotopy class of a straight line segment connecting the preimages of $x$ and $y$. We call this a \textbf{bridging or vertical chord}. If the bridging arc begins on the outer (inner) boundary component, the representative of the isotopy class in the universal cover that is a straight line segment will have positive (negative) slope by Lemma \ref{Shifting Lemma}. Moreover, the preimage of the arc $a(x,y)[\lambda]$ will have length $l$ such that $\sqrt{2}|\lambda|<l<\sqrt{2}(|\lambda|+1)$. The preimage of an exterior arc $a(x,y)$ under the covering map $p$ is the isotopy class of simple curves connecting the preimages of $x$ and $y$. We call this an \textbf{exterior or boundary chord}. A \textbf{triangulation} of the cover of an annulus is the preimage of a triangulation of the annulus under the covering map $p$.

We now fix a triangulation $T$ of the cover consisting entirely of bridging chords $X_{i \bmod p}$, labeling each consecutively from left to right with $i$, except we label $X_0$ by $X_p$. We will call these fixed vertical chords the \textbf{steep} chords. We will call the image of steep chords under the covering map \textbf{steep arcs}. Note that steep arcs must be bridging. The vertex set of $Q^{\bm{\varepsilon}}$ is now given by the equivalence classes of steep chords $Q_0^{\bm{\varepsilon}} = \{1, 2, \dots, p\}$ where two steep chords are equivalent if and only if they have the same label. Now if $X_j$ is clockwise of $X_{j+1}$, we have that $\varepsilon_j = -$. On the other hand, if $X_j$ is counter-clockwise of $X_{j+1}$, we have that $\varepsilon_j = +$. We now have a quiver $Q^{\bm{\varepsilon}}$ and a cluster algebra $\mathscr{A}_{Q^{\bm{\varepsilon}}}$ associated to the triangulation of $A_{Q^{\bm{\varepsilon}}}$ by steep arcs.

\begin{exmp}\label{exmp: steep triangulation}
Below is an example of a triangulation of an annulus $A_{2,2}$ by steep arcs and its lift to a triangulation of the universal cover by steep chords.
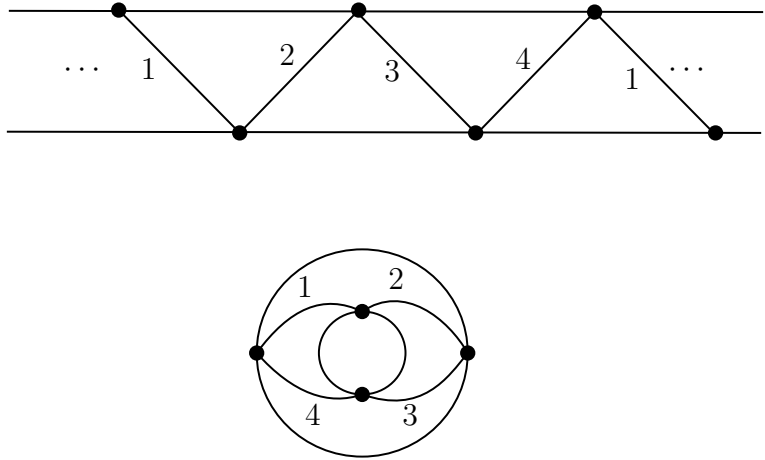
\begin{center}

\tikzset{every picture/.style={line width=0.75pt}} 

\begin{tikzpicture}[x=0.75pt,y=0.75pt,yscale=-1,xscale=1]

\draw   (300.38,192.7) .. controls (300.38,181.14) and (310.16,171.77) .. (322.22,171.77) .. controls (334.28,171.77) and (344.06,181.14) .. (344.06,192.7) .. controls (344.06,204.25) and (334.28,213.62) .. (322.22,213.62) .. controls (310.16,213.62) and (300.38,204.25) .. (300.38,192.7)(269,192.7) .. controls (269,163.81) and (292.83,140.39) .. (322.22,140.39) .. controls (351.61,140.39) and (375.44,163.81) .. (375.44,192.7) .. controls (375.44,221.58) and (351.61,245) .. (322.22,245) .. controls (292.83,245) and (269,221.58) .. (269,192.7) ;
\draw    (322.22,171.77) ;
\draw [shift={(322.22,171.77)}, rotate = 0] [color={rgb, 255:red, 0; green, 0; blue, 0 }  ][fill={rgb, 255:red, 0; green, 0; blue, 0 }  ][line width=0.75]      (0, 0) circle [x radius= 3.35, y radius= 3.35]   ;
\draw    (322.22,213.62) ;
\draw [shift={(322.22,213.62)}, rotate = 0] [color={rgb, 255:red, 0; green, 0; blue, 0 }  ][fill={rgb, 255:red, 0; green, 0; blue, 0 }  ][line width=0.75]      (0, 0) circle [x radius= 3.35, y radius= 3.35]   ;
\draw    (375.44,192.7) ;
\draw [shift={(375.44,192.7)}, rotate = 0] [color={rgb, 255:red, 0; green, 0; blue, 0 }  ][fill={rgb, 255:red, 0; green, 0; blue, 0 }  ][line width=0.75]      (0, 0) circle [x radius= 3.35, y radius= 3.35]   ;
\draw    (269,192.7) ;
\draw [shift={(269,192.7)}, rotate = 0] [color={rgb, 255:red, 0; green, 0; blue, 0 }  ][fill={rgb, 255:red, 0; green, 0; blue, 0 }  ][line width=0.75]      (0, 0) circle [x radius= 3.35, y radius= 3.35]   ;
\draw [color={rgb, 255:red, 0; green, 0; blue, 0 }  ,draw opacity=1 ]   (269,192.7) .. controls (285.44,170.39) and (302.44,162.39) .. (322.22,171.77) ;
\draw [color={rgb, 255:red, 0; green, 0; blue, 0 }  ,draw opacity=1 ]   (269,192.7) .. controls (287.44,213.39) and (305.44,219.39) .. (322.22,213.62) ;
\draw [color={rgb, 255:red, 0; green, 0; blue, 0 }  ,draw opacity=1 ]   (322.22,171.77) .. controls (343.44,157.39) and (365.44,175.39) .. (375.44,192.7) ;
\draw [color={rgb, 255:red, 0; green, 0; blue, 0 }  ,draw opacity=1 ]   (375.44,192.7) .. controls (365.44,205.18) and (352.44,224.18) .. (322.22,213.62) ;
\draw    (144,20) -- (524.44,20.62) ;
\draw    (143,81) -- (523.44,81.62) ;
\draw    (260.44,81.62) ;
\draw [shift={(260.44,81.62)}, rotate = 0] [color={rgb, 255:red, 0; green, 0; blue, 0 }  ][fill={rgb, 255:red, 0; green, 0; blue, 0 }  ][line width=0.75]      (0, 0) circle [x radius= 3.35, y radius= 3.35]   ;
\draw    (320.44,19.62) ;
\draw [shift={(320.44,19.62)}, rotate = 0] [color={rgb, 255:red, 0; green, 0; blue, 0 }  ][fill={rgb, 255:red, 0; green, 0; blue, 0 }  ][line width=0.75]      (0, 0) circle [x radius= 3.35, y radius= 3.35]   ;
\draw    (379.44,81.62) ;
\draw [shift={(379.44,81.62)}, rotate = 0] [color={rgb, 255:red, 0; green, 0; blue, 0 }  ][fill={rgb, 255:red, 0; green, 0; blue, 0 }  ][line width=0.75]      (0, 0) circle [x radius= 3.35, y radius= 3.35]   ;
\draw    (439.44,20.62) ;
\draw [shift={(439.44,20.62)}, rotate = 0] [color={rgb, 255:red, 0; green, 0; blue, 0 }  ][fill={rgb, 255:red, 0; green, 0; blue, 0 }  ][line width=0.75]      (0, 0) circle [x radius= 3.35, y radius= 3.35]   ;
\draw    (500.44,81.62) ;
\draw [shift={(500.44,81.62)}, rotate = 0] [color={rgb, 255:red, 0; green, 0; blue, 0 }  ][fill={rgb, 255:red, 0; green, 0; blue, 0 }  ][line width=0.75]      (0, 0) circle [x radius= 3.35, y radius= 3.35]   ;
\draw    (199.44,19.62) ;
\draw [shift={(199.44,19.62)}, rotate = 0] [color={rgb, 255:red, 0; green, 0; blue, 0 }  ][fill={rgb, 255:red, 0; green, 0; blue, 0 }  ][line width=0.75]      (0, 0) circle [x radius= 3.35, y radius= 3.35]   ;
\draw    (199.44,19.62) -- (260.44,81.62) ;
\draw    (318.44,19.62) -- (379.44,81.62) ;
\draw    (439.44,20.62) -- (500.44,82.62) ;
\draw    (260.44,81.62) -- (320.44,19.62) ;
\draw    (379.44,82.62) -- (439.44,20.62) ;

\draw (288,152.4) node [anchor=north west][inner sep=0.75pt]    {$1$};
\draw (334,148.4) node [anchor=north west][inner sep=0.75pt]    {$2$};
\draw (341,217.4) node [anchor=north west][inner sep=0.75pt]    {$3$};
\draw (292,217.4) node [anchor=north west][inner sep=0.75pt]    {$4$};
\draw (209,42.4) node [anchor=north west][inner sep=0.75pt]    {$1$};
\draw (170,45) node [anchor=north west][inner sep=0.75pt]    {$\cdots$};
\draw (453,47.4) node [anchor=north west][inner sep=0.75pt]    {$1$};
\draw (475,45) node [anchor=north west][inner sep=0.75pt]    {$\cdots$};
\draw (279,35.4) node [anchor=north west][inner sep=0.75pt]    {$2$};
\draw (332,43.4) node [anchor=north west][inner sep=0.75pt]    {$3$};
\draw (398,37.4) node [anchor=north west][inner sep=0.75pt]    {$4$};

\end{tikzpicture}

\end{center}

With our labeling convention, this triangulation gives the quiver $$\xymatrix{1 \ar[r] \ar@/^1.5pc/[rrr] & 2 & 3 \ar[l] \ar[r] & 4}.$$

\end{exmp}

Now consider any arc $a(x,y)[\lambda]$ on the annulus which does not self intersect nontrivially. This arc lifts to a chord in the universal cover which either:
\begin{enumerate}
    \item intersects some finite subset of steep chords,
    \item is one of the $p$ steep chords: $X_1,X_2,\dots, X_p$, or
    \item is homotopic to a piece of the boundary component of the cover.
\end{enumerate} 

We will now define a map $\varphi$ that associates to each chord in the universal cover of $A_{Q^{\bm\varepsilon}}$, a $\Bbbk Q^{\bm{\varepsilon}}$-module as follows. If it is a steep chord or homotopic to a piece of a boundary, $\varphi$ assigns to it the $0$ module. If it is not, suppose that when ordered from left to right the subset of vertical chords the lifted arc intersects is $\{X_i, X_{i+1}, \dots, X_j\}$ where subscripts are taken modulo $p$, where instead of $X_0$, we write $X_p$. Then to this lifted arc, $\varphi$ assigns the string module $ij_{j-i+1}$. As was shown in \cite{Master'sThesisClustersandTriangulations}, $\varphi$ forms a bijection between non-steep arcs on the annulus that are not homotopic to the boundary and indecomposable $\Bbbk Q^{\bm{\varepsilon}}$-modules.

Using the aforementioned bijection, we can get a cluster from any triangulation of the annulus and any triangulation of the annulus from a cluster as follows. An arbitrary triangulation of the annulus will consist of $l\ge0$ steep arcs that lift to steep chords say $\{X_{j_1}, X_{j_2}, \dots, X_{j_l}\}$, and $p - l$ arcs that don't lift to steep chords, say $\{Y_{l+1}, Y_{l+2}, \dots, Y_{p}\}$. The cluster corresponding to this triangulation is $\varphi(Y_{l+1}) \oplus \dots \oplus \varphi(Y_{n+1}) \oplus \Sigma P(j_1)\ \oplus \dots \oplus \Sigma P(j_l)$ where $\Sigma P(i)$ is the shifted projective at vertex $i$ in the cluster category $\mathcal{C}_{Q^{\bm{\varepsilon}}}$. As shown in \cite{Master'sThesisClustersandTriangulations}, this association is a bijection. 

\begin{exmp}
Below is an example of a triangulation of the annulus, its lift to the universal cover, and its corresponding cluster in $\mathcal{C}_{Q^{\bm{\varepsilon}}}$ from Example \ref{exmp: steep triangulation}.

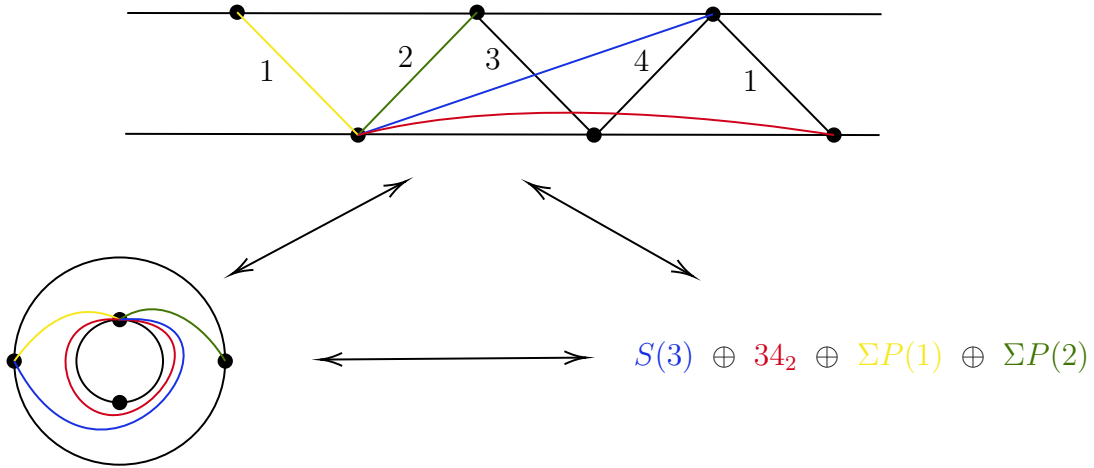
\begin{center}

\tikzset{every picture/.style={line width=0.75pt}} 

\begin{tikzpicture}[x=0.75pt,y=0.75pt,yscale=-1,xscale=1]

\draw   (118.38,195.7) .. controls (118.38,184.14) and (128.16,174.77) .. (140.22,174.77) .. controls (152.28,174.77) and (162.06,184.14) .. (162.06,195.7) .. controls (162.06,207.25) and (152.28,216.62) .. (140.22,216.62) .. controls (128.16,216.62) and (118.38,207.25) .. (118.38,195.7)(87,195.7) .. controls (87,166.81) and (110.83,143.39) .. (140.22,143.39) .. controls (169.61,143.39) and (193.44,166.81) .. (193.44,195.7) .. controls (193.44,224.58) and (169.61,248) .. (140.22,248) .. controls (110.83,248) and (87,224.58) .. (87,195.7) ;
\draw    (140.22,174.77) ;
\draw [shift={(140.22,174.77)}, rotate = 0] [color={rgb, 255:red, 0; green, 0; blue, 0 }  ][fill={rgb, 255:red, 0; green, 0; blue, 0 }  ][line width=0.75]      (0, 0) circle [x radius= 3.35, y radius= 3.35]   ;
\draw    (140.22,216.62) ;
\draw [shift={(140.22,216.62)}, rotate = 0] [color={rgb, 255:red, 0; green, 0; blue, 0 }  ][fill={rgb, 255:red, 0; green, 0; blue, 0 }  ][line width=0.75]      (0, 0) circle [x radius= 3.35, y radius= 3.35]   ;
\draw    (193.44,195.7) ;
\draw [shift={(193.44,195.7)}, rotate = 0] [color={rgb, 255:red, 0; green, 0; blue, 0 }  ][fill={rgb, 255:red, 0; green, 0; blue, 0 }  ][line width=0.75]      (0, 0) circle [x radius= 3.35, y radius= 3.35]   ;
\draw    (87,195.7) ;
\draw [shift={(87,195.7)}, rotate = 0] [color={rgb, 255:red, 0; green, 0; blue, 0 }  ][fill={rgb, 255:red, 0; green, 0; blue, 0 }  ][line width=0.75]      (0, 0) circle [x radius= 3.35, y radius= 3.35]   ;
\draw [color={rgb, 255:red, 248; green, 231; blue, 28 }  ,draw opacity=1 ]   (87,195.7) .. controls (103.44,173.39) and (120.44,165.39) .. (140.22,174.77) ;
\draw [color={rgb, 255:red, 208; green, 2; blue, 27 }  ,draw opacity=1 ]   (140.22,174.77) .. controls (100.74,171.29) and (108.44,220.94) .. (135.44,222.94) .. controls (162.44,224.94) and (190.44,174.94) .. (140.22,174.77) -- cycle ;
\draw [color={rgb, 255:red, 65; green, 117; blue, 5 }  ,draw opacity=1 ]   (140.22,174.77) .. controls (161.44,160.39) and (183.44,178.39) .. (193.44,195.7) ;
\draw [color={rgb, 255:red, 22; green, 48; blue, 226 }  ,draw opacity=1 ]   (140.22,174.77) .. controls (219.44,167.94) and (133.44,286.94) .. (87,195.7) ;
\draw    (144,20) -- (524.44,20.62) ;
\draw    (143,81) -- (523.44,81.62) ;
\draw    (260.44,81.62) ;
\draw [shift={(260.44,81.62)}, rotate = 0] [color={rgb, 255:red, 0; green, 0; blue, 0 }  ][fill={rgb, 255:red, 0; green, 0; blue, 0 }  ][line width=0.75]      (0, 0) circle [x radius= 3.35, y radius= 3.35]   ;
\draw    (320.44,19.62) ;
\draw [shift={(320.44,19.62)}, rotate = 0] [color={rgb, 255:red, 0; green, 0; blue, 0 }  ][fill={rgb, 255:red, 0; green, 0; blue, 0 }  ][line width=0.75]      (0, 0) circle [x radius= 3.35, y radius= 3.35]   ;
\draw    (379.44,81.62) ;
\draw [shift={(379.44,81.62)}, rotate = 0] [color={rgb, 255:red, 0; green, 0; blue, 0 }  ][fill={rgb, 255:red, 0; green, 0; blue, 0 }  ][line width=0.75]      (0, 0) circle [x radius= 3.35, y radius= 3.35]   ;
\draw    (439.44,20.62) ;
\draw [shift={(439.44,20.62)}, rotate = 0] [color={rgb, 255:red, 0; green, 0; blue, 0 }  ][fill={rgb, 255:red, 0; green, 0; blue, 0 }  ][line width=0.75]      (0, 0) circle [x radius= 3.35, y radius= 3.35]   ;
\draw    (500.44,81.62) ;
\draw [shift={(500.44,81.62)}, rotate = 0] [color={rgb, 255:red, 0; green, 0; blue, 0 }  ][fill={rgb, 255:red, 0; green, 0; blue, 0 }  ][line width=0.75]      (0, 0) circle [x radius= 3.35, y radius= 3.35]   ;
\draw    (199.44,19.62) ;
\draw [shift={(199.44,19.62)}, rotate = 0] [color={rgb, 255:red, 0; green, 0; blue, 0 }  ][fill={rgb, 255:red, 0; green, 0; blue, 0 }  ][line width=0.75]      (0, 0) circle [x radius= 3.35, y radius= 3.35]   ;
\draw [color={rgb, 255:red, 248; green, 231; blue, 28 }  ,draw opacity=1 ]   (199.44,19.62) -- (260.44,81.62) ;
\draw    (318.44,19.62) -- (379.44,81.62) ;
\draw    (439.44,20.62) -- (500.44,82.62) ;
\draw [color={rgb, 255:red, 65; green, 117; blue, 5 }  ,draw opacity=1 ]   (260.44,81.62) -- (320.44,19.62) ;
\draw [color={rgb, 255:red, 0; green, 0; blue, 0 }  ,draw opacity=1 ]   (379.44,82.62) -- (439.44,20.62) ;
\draw [color={rgb, 255:red, 22; green, 48; blue, 226 }  ,draw opacity=1 ]   (260.44,81.62) -- (439.44,20.62) ;
\draw [color={rgb, 255:red, 208; green, 2; blue, 27 }  ,draw opacity=1 ]   (260.44,81.62) .. controls (343.44,59.94) and (458.44,74.94) .. (500.44,81.62) ;
\draw    (243.44,194.93) -- (373.44,193.96) ;
\draw [shift={(375.44,193.94)}, rotate = 179.57] [color={rgb, 255:red, 0; green, 0; blue, 0 }  ][line width=0.75]    (10.93,-3.29) .. controls (6.95,-1.4) and (3.31,-0.3) .. (0,0) .. controls (3.31,0.3) and (6.95,1.4) .. (10.93,3.29)   ;
\draw [shift={(241.44,194.94)}, rotate = 359.57] [color={rgb, 255:red, 0; green, 0; blue, 0 }  ][line width=0.75]    (10.93,-3.29) .. controls (6.95,-1.4) and (3.31,-0.3) .. (0,0) .. controls (3.31,0.3) and (6.95,1.4) .. (10.93,3.29)   ;
\draw    (198.2,151) -- (282.67,105.88) ;
\draw [shift={(284.44,104.94)}, rotate = 151.89] [color={rgb, 255:red, 0; green, 0; blue, 0 }  ][line width=0.75]    (10.93,-3.29) .. controls (6.95,-1.4) and (3.31,-0.3) .. (0,0) .. controls (3.31,0.3) and (6.95,1.4) .. (10.93,3.29)   ;
\draw [shift={(196.44,151.94)}, rotate = 331.89] [color={rgb, 255:red, 0; green, 0; blue, 0 }  ][line width=0.75]    (10.93,-3.29) .. controls (6.95,-1.4) and (3.31,-0.3) .. (0,0) .. controls (3.31,0.3) and (6.95,1.4) .. (10.93,3.29)   ;
\draw    (348.18,106.93) -- (427.7,151.95) ;
\draw [shift={(429.44,152.94)}, rotate = 209.52] [color={rgb, 255:red, 0; green, 0; blue, 0 }  ][line width=0.75]    (10.93,-3.29) .. controls (6.95,-1.4) and (3.31,-0.3) .. (0,0) .. controls (3.31,0.3) and (6.95,1.4) .. (10.93,3.29)   ;
\draw [shift={(346.44,105.94)}, rotate = 29.52] [color={rgb, 255:red, 0; green, 0; blue, 0 }  ][line width=0.75]    (10.93,-3.29) .. controls (6.95,-1.4) and (3.31,-0.3) .. (0,0) .. controls (3.31,0.3) and (6.95,1.4) .. (10.93,3.29)   ;

\draw (209,42.4) node [anchor=north west][inner sep=0.75pt]    {$1$};
\draw (453,47.4) node [anchor=north west][inner sep=0.75pt]    {$1$};
\draw (279,35.4) node [anchor=north west][inner sep=0.75pt]    {$2$};
\draw (323,36.4) node [anchor=north west][inner sep=0.75pt]    {$3$};
\draw (398,37.4) node [anchor=north west][inner sep=0.75pt]    {$4$};
\draw (398,184.4) node [anchor=north west][inner sep=0.75pt]    {$\textcolor[rgb]{0.09,0.19,0.89}{S(3)} \ \oplus \ \textcolor[rgb]{0.82,0.01,0.11}{34_{2}} \ \oplus \ \textcolor[rgb]{0.97,0.91,0.11}{\Sigma P( 1)} \ \oplus \ \textcolor[rgb]{0.25,0.46,0.02}{\Sigma P( 2)}$};

\end{tikzpicture}

\end{center}
\end{exmp}

\section{Counting Method}\label{sec: counting method}

In this section, we will place clusters into infinite families using their triangulations and count how many such families there are.

\begin{defn}
    We say that two triangulations of the annulus $A_{n,m}$ are \textbf{inner equivalent} if and only if they differ by a sequence of $2\pi$ Dehn twists of the inner (equivalently outer) boundary component of the annulus. We will also refer to inner equivalence classes as \textbf{families} throughout the paper.
\end{defn}

Our goal for this section is to prove the following theorem.

\begin{thm}\label{finalresult} $\textbf{}$ \\
    The number of inner equivalence classes of triangulations of an annulus with $n$ vertices on the outside boundary component and $m$ vertices on the inside boundary component is $T(A_{n,m})=n\big(\sum_{i=0}^{m-1}(i+1)C_iC_{n+m-i-1})$ where $C_k=\frac{1}{k+1}\frac{(2k)!}{k!k!}$ is the $k^{th}$ Catalan number.
\end{thm}

\begin{rem}
    Note that in the case $m=1$, the formula in Theorem \ref{finalresult} becomes $nC_n$ after applying the convolution formula for the Catalan numbers, which agrees with the formula found in \cite{IgusaMaresca}.
\end{rem}

In order to count the number of inner equivalence classes of triangulations, we must find a unique representative of each class. To do this, we begin with a lemma.

\begin{lem}\label{cor: existence of bridging arc}
    Let $t$ be a triangulation of $A_{n,m}$. Then there exist at least two bridging arcs contained in $t$.  
\end{lem}
\begin{proof}
    Consider the annulus $A_{n,m}$. To prove this statement we will show that there are at most $n+m - 2$ boundary arcs in any triangulation. To do this, we begin constructing a triangulation of this annulus by placing the maximal number of small boundary arcs on the outer boundary component that can occur in a triangulation. Since these boundary arcs do not cross, there exists an endpoint $a_j$ such that we can place a loop; that is, a boundary arc $a(a_j,a_j)[1]$ that does not cross any other boundary arc, hence is allowed in the triangulation. We conclude that to maximize the total number of boundary arcs on the outer component, we must have one loop. To compute the maximal number of boundary arcs, we thus without loss of generality fix a loop at vertex $a_1$. Then triangulating the outer boundary component with small boundary arcs amounts to triangulating an $n+1$-gon with vertices consecutively labeled $a_1,a_1,a_2,a_3,\dots,a_n$. There are $n-2$ diagonals in any such triangulation and therefore we conclude that the maximal number of boundary arcs in any triangulation is $n-1$. The same proof holds for the the inner boundary component. Therefore there are at most $n+m-2$ boundary arcs in any triangulation of $A_{n,m}$. Since a triangulation of this annulus contains $n+m$ arcs, we have the result.  
\end{proof}

The following lemma was proven in \cite{IgusaMaresca}, however, here we provide a nicer proof.

\begin{lem}\label{smalltriangulation}
    In every family of triangulations $[t]$ there exists a small triangulation.
\end{lem}
\begin{proof}
 Pick an arbitrary triangulation $t\in [t]$. By Lemma \ref{cor: existence of bridging arc}, there exists a bridging arc. Pick an arbitrary bridging arc and apply $2\pi$ Dehn twists until we obtain a new triangulation in which this arc has winding number 0. If this triangulation is small we are done, if not, there exists an additional bridging arc $a(a,b)$ with winding number $\pm 1$. Consider the preimage of this triangulation on the universal cover. 
    
    Using Lemma \ref{Shifting Lemma} this bridging arc with winding number $\pm 1$ is homotopic to a line of either a positive or negative slope with length greater than $\sqrt{2}$ and less than $2\sqrt{2}$ on the universal cover. Suppose its slope is positive, so the winding number is 1, as the case where it is negative is proved similarly. 

    Consider the subset of the universal cover lying entirely between two consecutive lifts of the bridging arc $a(a,b)$. There are two representatives of $b$ in this parallelogram, $b'$ and $b''=b'+(1,0)$, and two representatives of $a$, namely $a'$ and $a''=a'+(1,0)$. Note that since the winding number of $a(a,b)$ is $1$, we have $(b'')_x<(a')_x$. Consider the line in this parallelogram that is the preimage of an arbitrary bridging arc connecting the marked points $x,y$ in the triangulation containing $a(a,b)$. The representative of $x$ is then to the left of $b''$, and the representative of $y$ is to the right of $a'$, so that our line has positive slope, so that this bridging arc is given by $a(x,y)[0]$.  
    
    Applying a counter-clockwise $2\pi$ Dehn twist of the inner boundary component (in the negative case this would be a clockwise $2\pi$ Dehn twist of the inner boundary component), $a(a,b)[1] \mapsto a(a,b)[0]$, while $a(x,y)[0] \mapsto a(y,x)[0]$. Since $a(x,y)$ was arbitrary, the winding number of all arcs of the triangulation obtained by this counter-clockwise $2\pi$ Dehn twist is 0.
\end{proof}

We will use the convention of calling $b_m $ the marked point on the inside circle that is immediately counter-clockwise from $a_1$. Formally, after fixing a representative $a'_1\in \mathbb{R} \times [0,1]$ of $a_1$, pick a representative $b'_m$ which is immediately to the left of $a'_1$ which is possible by the natural ordering of the reals, and denote $b_m=p(b'_m)$. Inductively we now have a naming as before, where $b_{(i+1)\bmod m} \vcentcolon = Suc(b_{(i)\bmod m})$ is the vertex immediately clockwise from $b_{i}$. We will also introduce the notion of a fundamental date line that will play a key role in our proofs. We construct it as follows. On the outer boundary component we place a point $a$ between the marked points $a_0=a_n$ and $a_1$. Similarly, we place a point $b$ on the inner boundary component between the marked points $b_m=b_0$ and $b_1$. We can choose them to be equidistant from the marked points they were placed in between so that on the universal cover their cosets correspond to $\{\frac{0.5}{n}+\mathbb{Z}\} \times \{1\}$ and $\{\frac{0.5}{m}+\mathbb{Z}\}\times \{0\} $ respectively. We call the line connecting these marked points the \textbf{fundamental date line} (fdl). 
On the annulus, for $1\leq i \leq n,1\leq j \leq m, 0\leq k \leq m-1$ such that $j\leq (m-k)\bmod m$, we can now define the bridging arc $\alpha_{m-k}^i$ to be the clockwise arc connecting the vertex $b_{m-k}$ with $a_i$ thus crossing the fdl once, the bridging arc $\beta_j^i$ to be the arc with endpoints $a_i$ and $b_j$ not crossing the fdl, and the boundary arc $\gamma_{m-k,j}$ to be the arc connecting $b_{m-k}$ with $b_j$. Formally, in light of Definition \ref{defn: arcs on annulus}, we have that $\alpha_{m-k}^i=a(b_{m-k},a_i)$. On the other hand, $\beta_{j}^i$ is either $a(b_j,a_i)$ or $a(a_i,b_j)$, depending on which one of those does not cross the fdl. Since they form a loop one of them will always cross the fdl, and the other will not. The exterior arc is always of the form $\gamma_{m-k,j}=a(b_{m-k},b_j)$. We denote by $S_{(a_i,b_{m-k},b_j)}$ the set of all triangulations containing arcs $\alpha_{m-k}^i,\beta_{j}^i,\gamma_{m-k,j}$, which we will in turn call the \textbf{defining arcs} of that set since they are uniquely determined by the choice of $i,j,k$. It is important to note that the order of the vertices in the notation of $S_{(-,-,-)}$ is important. The arc whose end points are the first and second vertices is the one crossing the fdl.

\begin{rem}
    Note that the defining arcs $\alpha$ have winding number either 0 or 1, while the arcs $\beta$ always have winding number 0. Also, $\gamma$ has winding number 1 if and only if it is a loop, otherwise it has winding number 0.
\end{rem}


\begin{exmp}
    If we let $n=4, m=5$, the set $S_{(a_2,b_3,b_1)}$ will contain the triangulations containing the blue and orange arcs (as well as the unique implicit arc between $b_3$ and $b_1$ such that it forms a triangle with the blue and orange arc) in the following picture The dashed red line represents the fdl. The cardinality of $S_{(a_2,b_3,b_1)}$ will count the number of triangulations that contain the blue and orange arcs. 
    
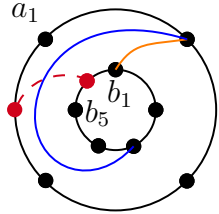
\begin{center}
\tikzset{every picture/.style={line width=0.75pt}} 

\begin{tikzpicture}[x=0.75pt,y=0.75pt,yscale=-1,xscale=1]

\draw   (87.78,97.47) .. controls (87.78,86.26) and (96.86,77.17) .. (108.07,77.17) .. controls (119.28,77.17) and (128.37,86.26) .. (128.37,97.47) .. controls (128.37,108.68) and (119.28,117.76) .. (108.07,117.76) .. controls (96.86,117.76) and (87.78,108.68) .. (87.78,97.47)(57.33,97.47) .. controls (57.33,69.45) and (80.05,46.73) .. (108.07,46.73) .. controls (136.09,46.73) and (158.81,69.45) .. (158.81,97.47) .. controls (158.81,125.49) and (136.09,148.2) .. (108.07,148.2) .. controls (80.05,148.2) and (57.33,125.49) .. (57.33,97.47) ;
\draw    (72.69,61.84) ;
\draw [shift={(72.69,61.84)}, rotate = 0] [color={rgb, 255:red, 0; green, 0; blue, 0 }  ][fill={rgb, 255:red, 0; green, 0; blue, 0 }  ][line width=0.75]      (0, 0) circle [x radius= 3.35, y radius= 3.35]   ;
\draw [shift={(72.69,61.84)}, rotate = 0] [color={rgb, 255:red, 0; green, 0; blue, 0 }  ][fill={rgb, 255:red, 0; green, 0; blue, 0 }  ][line width=0.75]      (0, 0) circle [x radius= 3.35, y radius= 3.35]   ;
\draw    (144.17,61.84) ;
\draw [shift={(144.17,61.84)}, rotate = 0] [color={rgb, 255:red, 0; green, 0; blue, 0 }  ][fill={rgb, 255:red, 0; green, 0; blue, 0 }  ][line width=0.75]      (0, 0) circle [x radius= 3.35, y radius= 3.35]   ;
\draw [shift={(144.17,61.84)}, rotate = 0] [color={rgb, 255:red, 0; green, 0; blue, 0 }  ][fill={rgb, 255:red, 0; green, 0; blue, 0 }  ][line width=0.75]      (0, 0) circle [x radius= 3.35, y radius= 3.35]   ;
\draw [color={rgb, 255:red, 0; green, 0; blue, 0 }  ,draw opacity=1 ]   (87.78,97.47) ;
\draw [shift={(87.78,97.47)}, rotate = 0] [color={rgb, 255:red, 0; green, 0; blue, 0 }  ,draw opacity=1 ][fill={rgb, 255:red, 0; green, 0; blue, 0 }  ,fill opacity=1 ][line width=0.75]      (0, 0) circle [x radius= 3.35, y radius= 3.35]   ;
\draw [shift={(87.78,97.47)}, rotate = 0] [color={rgb, 255:red, 0; green, 0; blue, 0 }  ,draw opacity=1 ][fill={rgb, 255:red, 0; green, 0; blue, 0 }  ,fill opacity=1 ][line width=0.75]      (0, 0) circle [x radius= 3.35, y radius= 3.35]   ;
\draw    (108.07,77.17) ;
\draw [shift={(108.07,77.17)}, rotate = 0] [color={rgb, 255:red, 0; green, 0; blue, 0 }  ][fill={rgb, 255:red, 0; green, 0; blue, 0 }  ][line width=0.75]      (0, 0) circle [x radius= 3.35, y radius= 3.35]   ;
\draw [shift={(108.07,77.17)}, rotate = 0] [color={rgb, 255:red, 0; green, 0; blue, 0 }  ][fill={rgb, 255:red, 0; green, 0; blue, 0 }  ][line width=0.75]      (0, 0) circle [x radius= 3.35, y radius= 3.35]   ;
\draw    (128.37,97.47) ;
\draw [shift={(128.37,97.47)}, rotate = 0] [color={rgb, 255:red, 0; green, 0; blue, 0 }  ][fill={rgb, 255:red, 0; green, 0; blue, 0 }  ][line width=0.75]      (0, 0) circle [x radius= 3.35, y radius= 3.35]   ;
\draw [shift={(128.37,97.47)}, rotate = 0] [color={rgb, 255:red, 0; green, 0; blue, 0 }  ][fill={rgb, 255:red, 0; green, 0; blue, 0 }  ][line width=0.75]      (0, 0) circle [x radius= 3.35, y radius= 3.35]   ;
\draw    (144.05,133.33) ;
\draw [shift={(144.05,133.33)}, rotate = 0] [color={rgb, 255:red, 0; green, 0; blue, 0 }  ][fill={rgb, 255:red, 0; green, 0; blue, 0 }  ][line width=0.75]      (0, 0) circle [x radius= 3.35, y radius= 3.35]   ;
\draw [shift={(144.05,133.33)}, rotate = 0] [color={rgb, 255:red, 0; green, 0; blue, 0 }  ][fill={rgb, 255:red, 0; green, 0; blue, 0 }  ][line width=0.75]      (0, 0) circle [x radius= 3.35, y radius= 3.35]   ;
\draw    (72.81,133.09) ;
\draw [shift={(72.81,133.09)}, rotate = 0] [color={rgb, 255:red, 0; green, 0; blue, 0 }  ][fill={rgb, 255:red, 0; green, 0; blue, 0 }  ][line width=0.75]      (0, 0) circle [x radius= 3.35, y radius= 3.35]   ;
\draw [shift={(72.81,133.09)}, rotate = 0] [color={rgb, 255:red, 0; green, 0; blue, 0 }  ][fill={rgb, 255:red, 0; green, 0; blue, 0 }  ][line width=0.75]      (0, 0) circle [x radius= 3.35, y radius= 3.35]   ;
\draw    (117.19,115.84) ;
\draw [shift={(117.19,115.84)}, rotate = 0] [color={rgb, 255:red, 0; green, 0; blue, 0 }  ][fill={rgb, 255:red, 0; green, 0; blue, 0 }  ][line width=0.75]      (0, 0) circle [x radius= 3.35, y radius= 3.35]   ;
\draw [shift={(117.19,115.84)}, rotate = 0] [color={rgb, 255:red, 0; green, 0; blue, 0 }  ][fill={rgb, 255:red, 0; green, 0; blue, 0 }  ][line width=0.75]      (0, 0) circle [x radius= 3.35, y radius= 3.35]   ;
\draw [color={rgb, 255:red, 255; green, 125; blue, 0 }  ,draw opacity=1 ]   (144.17,61.84) .. controls (124.17,64.64) and (117.2,62.8) .. (108.07,77.17) ;
\draw [color={rgb, 255:red, 0; green, 0; blue, 255 }  ,draw opacity=1 ]   (144.17,61.84) .. controls (86.7,44.31) and (65.21,80.83) .. (67.6,104) .. controls (70.04,127.66) and (97.19,137.95) .. (117.19,115.84) ;
\draw [color={rgb, 255:red, 0; green, 0; blue, 0 }  ,draw opacity=1 ]   (99.43,115.6) ;
\draw [shift={(99.43,115.6)}, rotate = 0] [color={rgb, 255:red, 0; green, 0; blue, 0 }  ,draw opacity=1 ][fill={rgb, 255:red, 0; green, 0; blue, 0 }  ,fill opacity=1 ][line width=0.75]      (0, 0) circle [x radius= 3.35, y radius= 3.35]   ;
\draw [shift={(99.43,115.6)}, rotate = 0] [color={rgb, 255:red, 0; green, 0; blue, 0 }  ,draw opacity=1 ][fill={rgb, 255:red, 0; green, 0; blue, 0 }  ,fill opacity=1 ][line width=0.75]      (0, 0) circle [x radius= 3.35, y radius= 3.35]   ;
\draw (110.4,89.2) node  {$b_1$};
\draw (62.7,48.5) node  {$a_1$};
\draw (99,99) node  {$b_5$};
\draw [color={rgb, 255:red, 208; green, 2; blue, 27 }  ,draw opacity=1 ][fill={rgb, 255:red, 208; green, 2; blue, 27 }  ,fill opacity=1 ]   (93.69,82.84) ;
\draw [shift={(93.69,82.84)}, rotate = 0] [color={rgb, 255:red, 208; green, 2; blue, 27 }  ,draw opacity=1 ][fill={rgb, 255:red, 208; green, 2; blue, 27 }  ,fill opacity=1 ][line width=0.75]      (0, 0) circle [x radius= 3.35, y radius= 3.35]   ;
\draw [shift={(93.69,82.84)}, rotate = 0] [color={rgb, 255:red, 208; green, 2; blue, 27 }  ,draw opacity=1 ][fill={rgb, 255:red, 208; green, 2; blue, 27 }  ,fill opacity=1 ][line width=0.75]      (0, 0) circle [x radius= 3.35, y radius= 3.35]   ;
\draw [color={rgb, 255:red, 208; green, 2; blue, 27 }  ,draw opacity=1 ] [dash pattern={on 4.5pt off 4.5pt}]  (93.69,82.84) .. controls (95.25,81) and (67.25,72.5) .. (57.33,97.47) ;
\draw [color={rgb, 255:red, 208; green, 2; blue, 27 }  ,draw opacity=1 ][fill={rgb, 255:red, 208; green, 2; blue, 27 }  ,fill opacity=1 ]   (57.19,97.34) ;
\draw [shift={(57.19,97.34)}, rotate = 0] [color={rgb, 255:red, 208; green, 2; blue, 27 }  ,draw opacity=1 ][fill={rgb, 255:red, 208; green, 2; blue, 27 }  ,fill opacity=1 ][line width=0.75]      (0, 0) circle [x radius= 3.35, y radius= 3.35];

\end{tikzpicture}
\end{center}
\end{exmp}

 It is possible to define these sets inductively. We start with $S_{(a_1,b_1,b_m)}$, which contains the arcs $\alpha = a(b_m,a_1)$ and $\beta = a(a_1,b_1)$. We can obtain from this $S_{(a_1,b_2,b_m)}$ by taking $\beta$ and changing it to become an arc from $a_1$ to $b_2$. The set $S_{(a_1,b_{m-1},b_1)}$ can be obtained similarly. The set $S_{(a_2,b_1,b_2)}$ is obtained from $S_{(a_1,b_1,b_2)}$ by a Dehn Twist of the outside circle. We can inductively continue. If we have $S_{(a_i,b_j,b_k)}$, $S_{(a_i,b_j,b_{k+1})}$ is obtained by removing whichever arc $a(a_i,b_k)$ or $a(b_k,a_i)$ in $S_{(a_i,b_j,b_k)}$ and adding the unique bridging arc connecting $a_i$ and $b_{k+1}$ which does not intersect, nor is homotopic to, whichever arc $a(a_i,b_j)$ or $a(b_j,a_i)$ is in  $S_{(a_i,b_j,b_k)}$. Similarly the set  $S_{(a_i,b_{j+1},b_k)}$ is obtained. Again, the set $S_{(a_{i+1},b_j,b_k)}$ is obtained by a Dehn twist of the outside boundary component. This iterative geometric process, which we will call our \textbf{counting method}, is illustrated in Figure $\ref{process}$. 

\begin{figure}[h!]
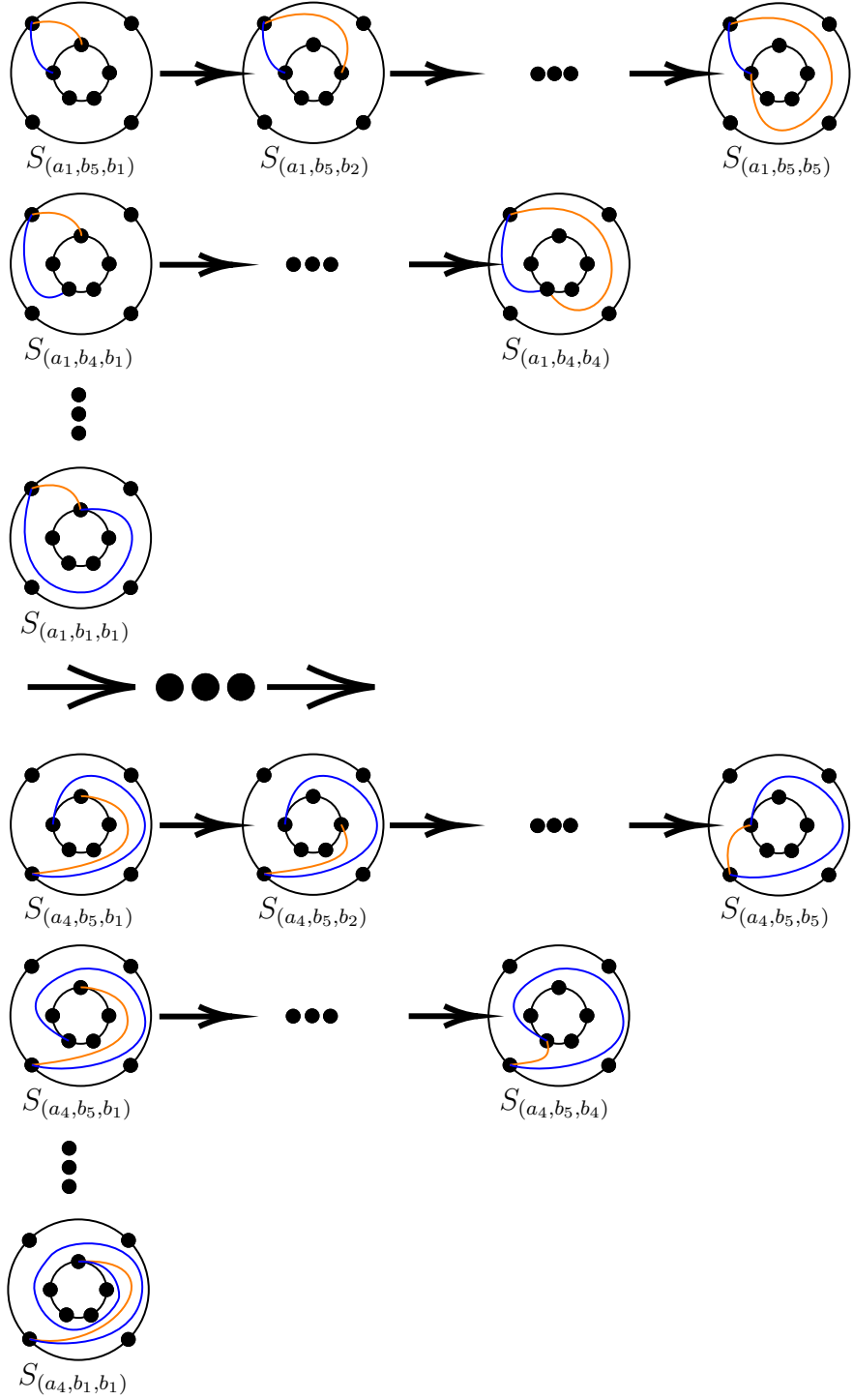

\begin{center}

\tikzset{every picture/.style={line width=0.75pt}} 



\end{center}
\caption{Constructions of the sets $S_{(-,-,-)}$}
\label{process}
\end{figure}

The set of all triangulations of an annulus is infinite, as even from a single triangulation, infinitely many can be generated by applying dehn twists.  However, in this paper we will show that modulo  $2\pi$ Dehn twists, there are finitely many families of triangulations. We denote by $[t]$ the family of triangulations obtained by $2\pi$ Dehn twists of a triangulation $t$. We will show that each family of triangulations has a unique representative lying in one of our sets $S_{(-,-,-)}$ constructed above so that the cardinality of the union of all of our constructed sets will count the number of families. These statements will be presented as individual lemmas, that when combined will prove the Theorem \ref{finalresult}. We begin with the following lemma.
\begin{lem}
The number of triangulations that this method counts is given by: \\  $n\big(\sum_{i=0}^{m-1}(i+1)C_iC_{n+m-i-1})$.
\end{lem}

\begin{proof}
Let us pick an arbitrary set $S_{(a_k,b_{m-i},b_{j+1})}$. We can visually interpret the set as the set of triangulations of Figure \ref{typingissue}. 

\tikzset{every picture/.style={line width=0.75pt}} 
\begin{figure}
\begin{tikzpicture}[x=0.75pt,y=0.75pt,yscale=-1,xscale=1]

\draw   (220.5,71.8) -- (508.2,71.8) -- (384.9,193.8) -- (97.2,193.8) -- cycle ;
\draw    (589,71.4) -- (104.6,71.4) -- (79,71.4) ;
\draw    (549,193.8) -- (39,194.2) ;
\draw    (220.5,71.8) ;
\draw [shift={(220.5,71.8)}, rotate = 0] [color={rgb, 255:red, 0; green, 0; blue, 0 }  ][fill={rgb, 255:red, 0; green, 0; blue, 0 }  ][line width=0.75]      (0, 0) circle [x radius= 2.34, y radius= 2.34]   ;
\draw [shift={(220.5,71.8)}, rotate = 0] [color={rgb, 255:red, 0; green, 0; blue, 0 }  ][fill={rgb, 255:red, 0; green, 0; blue, 0 }  ][line width=0.75]      (0, 0) circle [x radius= 2.34, y radius= 2.34]   ;
\draw    (508.2,71.8) ;
\draw [shift={(508.2,71.8)}, rotate = 0] [color={rgb, 255:red, 0; green, 0; blue, 0 }  ][fill={rgb, 255:red, 0; green, 0; blue, 0 }  ][line width=0.75]      (0, 0) circle [x radius= 2.34, y radius= 2.34]   ;
\draw [shift={(508.2,71.8)}, rotate = 0] [color={rgb, 255:red, 0; green, 0; blue, 0 }  ][fill={rgb, 255:red, 0; green, 0; blue, 0 }  ][line width=0.75]      (0, 0) circle [x radius= 2.34, y radius= 2.34]   ;
\draw    (97.2,193.8) ;
\draw [shift={(97.2,193.8)}, rotate = 0] [color={rgb, 255:red, 0; green, 0; blue, 0 }  ][fill={rgb, 255:red, 0; green, 0; blue, 0 }  ][line width=0.75]      (0, 0) circle [x radius= 2.34, y radius= 2.34]   ;
\draw [shift={(97.2,193.8)}, rotate = 0] [color={rgb, 255:red, 0; green, 0; blue, 0 }  ][fill={rgb, 255:red, 0; green, 0; blue, 0 }  ][line width=0.75]      (0, 0) circle [x radius= 2.34, y radius= 2.34]   ;
\draw    (384.9,193.8) ;
\draw [shift={(384.9,193.8)}, rotate = 0] [color={rgb, 255:red, 0; green, 0; blue, 0 }  ][fill={rgb, 255:red, 0; green, 0; blue, 0 }  ][line width=0.75]      (0, 0) circle [x radius= 2.34, y radius= 2.34]   ;
\draw [shift={(384.9,193.8)}, rotate = 0] [color={rgb, 255:red, 0; green, 0; blue, 0 }  ][fill={rgb, 255:red, 0; green, 0; blue, 0 }  ][line width=0.75]      (0, 0) circle [x radius= 2.34, y radius= 2.34]   ;
\draw (215.8,53.6) node  [anchor=north west][inner sep=0.75pt] {${a_k}$};
\draw (502.8,50) node  [anchor=north west][inner sep=0.75pt] {${a_k+(1,0)}$};
\draw (188.7,211.1) node {$b_m$};
\draw (200,142) node {$\gamma$};
\draw    (187.6,193.8) ;
\draw [shift={(187.6,193.8)}, rotate = 0] [color={rgb, 255:red, 0; green, 0; blue, 0 }  ][fill={rgb, 255:red, 0; green, 0; blue, 0 }  ][line width=0.75]      (0, 0) circle [x radius= 2.34, y radius= 2.34]   ;
\draw [shift={(187.6,193.8)}, rotate = 0] [color={rgb, 255:red, 0; green, 0; blue, 0 }  ][fill={rgb, 255:red, 0; green, 0; blue, 0 }  ][line width=0.75]      (0, 0) circle [x radius= 2.34, y radius= 2.34]   ;
\draw    (233.6,193.4) ;
\draw [shift={(233.6,193.4)}, rotate = 0] [color={rgb, 255:red, 0; green, 0; blue, 0 }  ][fill={rgb, 255:red, 0; green, 0; blue, 0 }  ][line width=0.75]      (0, 0) circle [x radius= 2.34, y radius= 2.34]   ;
\draw [shift={(233.6,193.4)}, rotate = 0] [color={rgb, 255:red, 0; green, 0; blue, 0 }  ][fill={rgb, 255:red, 0; green, 0; blue, 0 }  ][line width=0.75]      (0, 0) circle [x radius= 2.34, y radius= 2.34]   ;
\draw (232.4,210.2) node  {$b_1$};
\draw    (279.2,193.4) ;
\draw [shift={(279.2,193.4)}, rotate = 0] [color={rgb, 255:red, 0; green, 0; blue, 0 }  ][fill={rgb, 255:red, 0; green, 0; blue, 0 }  ][line width=0.75]      (0, 0) circle [x radius= 2.34, y radius= 2.34]   ;
\draw [shift={(279.2,193.4)}, rotate = 0] [color={rgb, 255:red, 0; green, 0; blue, 0 }  ][fill={rgb, 255:red, 0; green, 0; blue, 0 }  ][line width=0.75]      (0, 0) circle [x radius= 2.34, y radius= 2.34]   ;
\draw (285.7,209.4) node  {$b_{j+1}$};
\draw (95.5,209.7) node  {$b_{m-i}$};
\draw    (220.5,71.8) -- (279.2,193.4) ;
\draw    (97.2,193.8) .. controls (137.2,163.8) and (215.8,110.6) .. (279.2,193.4) ;
\draw (185,160) node [anchor=north west][inner sep=0.75pt]    {$R_1$};
\draw (340,125) node  {$R_2$};
89op 00

\end{tikzpicture}
    \centering
    \caption{Intersection of lifts of elements of $S_{(a_k,b_{m-i},b_{j+1})}$}
    \label{typingissue}
    
\end{figure}
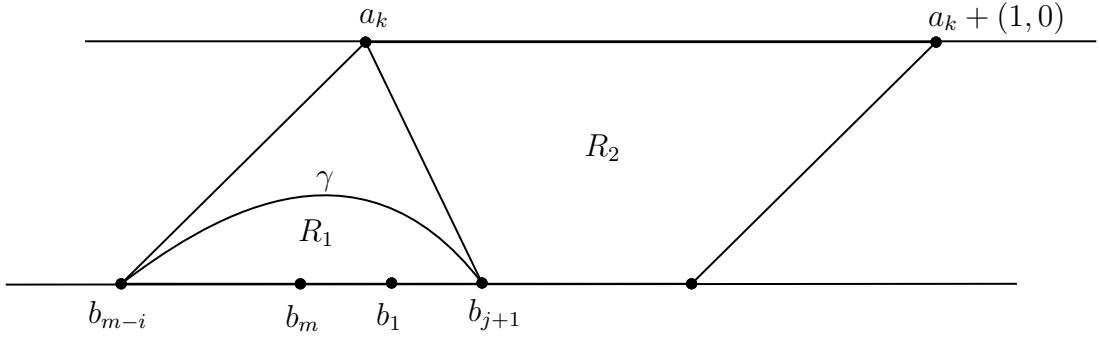

Note that it could be the case that $b_{m-i}=b_m$ and/or $b_{j+1}=b_1$, as well as that the cord connecting $a_k$ with $b_{j+1}$ might have positive slope. We must triangulate $R_1$ and $R_2$ in Figure \ref{typingissue}. We now triangulate the region bounded by the $\mathbb{R} \times {0}$ and $\gamma$. This region has $i+j+2$ vertices, thus this amounts to counting the triangulations of a regular $(i+j+2)$-gon \cite{IgusaMaresca}, which has $C_{i+j}$ triangulations, see for instance \cite{catalannumbers}. We now consider $R_2$. We have $n+1$ top vertices as the representative $a_k$ is counted twice. If we let $x$ denote the number of vertices on the bottom for the region $R_2$, we can write the equation $(i+1)+(j+1)+x=m+2$ as the vertices $b_{m-i}$ and $b_{j+1}$ are counted twice. Thus we have $m-i-j$ vertices on the bottom. Whence we have a polygon with a total of $n+m-i-j+1$ vertices, hence $C_{n+m-(i+j)-1}$ triangulations. Since $R_1$ and $R_2$ are disjoint, the total number of triangulations of the union of these regions is given by $C_{i+j}C_{n+m-(i+j)-1}$ \cite{ruleofprod}, Section 5.1.

For a fixed $a_k$, let us count how many sets of form $S_{(a_k,-,-)}$ there are. We need to put constrains on $i,j$ such that $S_{a_k,b_{m-i},b_{j+1}}$ is indeed one of our defined sets.

 Considering the vertices $b_{m-i}, b_{j+1}$, by the construction of our sets, we must have that $1\leq i+j+1 \leq m$, that is $0\leq i+j\leq m-1$. We can now construct the set $S$ of allowed tuples, $S:=\{ (i,j) \big| i+j\leq m-1 \} = \cup_{j=0}^{m-1} \cup_{i=0}^{m-j-1} \{(i,j)\}$, and for each tuple $(i,j)$ we can associate the set $S_{a_k,b_{m-i},b_{j+1}}$. Observe Lemma \ref{disjoint} gives that for distinct tuples $(i,j),(i',j')$ the sets $S_{a_k,b_{m-i},b_{j+1}},S_{a_k,b_{m-i'},b_{j'+1}}$ are disjoint. 
 
 Thus for a fixed $a_k$, iterating over the allowed $i,j$, our counting method counts \\
$\big| \cup_{j=0}^{m-1} \cup_{i=0}^{m-j-1} S_{(a_k,b_{m-i},b_{j+1})} \big|= \sum_{j=0}^{m-1} \sum_{i=0}^{m-j-1} C_{i+j}C_{n+m-(i+j)-1} = \sum_{j=0}^{m-1} \sum_{i=j}^{m-1} C_{i}C_{n+m-i-1}$\ many triangulations, where the first equality follows by our above observation, we used Lemma \ref{disjoint} in the second equality, and in the last equality we changed the index $i\rightarrow i-j$.  Finally we obtain $T=\sum_{j=0}^{m-1} \sum_{i=j}^{m-1} C_{i}C_{n+m-i-1}=\sum_{l=0}^{m-1}(l+1)C_lC_{n+m-l-1}$ by Lemma \ref{singlesum}. Notice that $T$ is independent of $k$, and since we have $n$ vertices on the outside, we obtain $T(A_{n,m}) = \sum_{1}^{n} T = nT$, our desired result.

\end{proof}

\begin{lem}\label{lem:surjectivity}
Given any triangulation $t$, then there exists a triangulation $t'\in [t]$ such that $\exists a_k,b_i,b_j$ with $t'\in S_{(a_k,b_i,b_j)}$.
\end{lem}

\begin{proof}
    Let $t$ be an arbitrary triangulation and consider its family $[t]$. We know by Lemma \ref{smalltriangulation} that there exists a triangulation in $t'\in[t]$ where all bridging arcs have winding number 0. We will prove that $t'$ contains three arcs $\alpha,\beta,\gamma$ which are the defining arcs of one of our sets, or that a $2\pi$ Dehn twist of $t'$, also a member of $[t]$, contains such three arcs. We begin by proving the existence of an arc $\gamma$. Note that $a(b_m,b_1)$ is always a boundary arc crossing fdl. Thus, since we have a finite number of arcs, we can consider the largest interior arc crossing the fdl, namely an arc connecting vertices $b_{m-k},b_j$ that crosses the fdl and minimizes $m-k-j$. In the case that $\gamma=a(b_m,b_1)$, then we consider $\gamma$ to be the part of the boundary that $a(b_m,b_1)$ is homotopic to. Thus we have found $\gamma_{m-k,j}$. Notice that $\gamma_{m-k,j}$ is invariant by a $2\pi$ Dehn twist. Similarly, we can construct the largest boundary arc crossing the fdl with both endpoints being on the outside boundary component. Suppose this arc is $a(a_l,a_k)$. Next, we prove that there is a bridging arc crossing the fdl. By contradiction, suppose there was no such arc. Consider the arcs $a(b_{m-k},a_l), a(a_k,b_j)$. By the maximality of a triangulation, these two arcs are only not allowed if there is already an arc crossing at least one of them, and such an arc would cross the fdl by construction of $a(b_{m-k},b_j)$ and $a(a_l,a_k)$. Hence, if these are not part of our triangulation, there exists a bridging arc crossing the fdl. If these arcs do exist, we still have the quadrilateral composed of vertices $b_{m-k},a_l,a_k,b_j$, contradicting the fact that we have a triangulation. Thus we must either have $a(b_{m-k},a_k)$ or $a(a_l,b_j)$, both of which are bridging arcs crossing the fdl.
    
    Now, we prove the existence of defining arcs $\alpha,\beta$. First notice that there is a bridging arc with endpoint $b_j$. If not, we could draw an interior arc crossing the fdl given by $a(b_{m-k},b_{j+1})$ contradicting the maximality of $\gamma$. Similarly, we obtain a bridging arc connecting to the vertex $b_{m-k}$. There are now two cases, either there is a bridging arc with an endpoint $b_{m-k}$ crossing the fdl, or there is not. Suppose first that there is. We can pick the one furthest to the right (clockwise direction), that is the arc $a(b_{m-k},a_i)$ maximizing $i$. Denote this arc by $\alpha=\alpha_{m-k}^i$. It is easy to see that we immediately now have $\beta=\beta_{i,j}$ be the arc $a(a_i,b_j)$, as the only way such an arc could not exist would bring forth a contradiction to the maximality of either $\gamma$ or $\alpha$. Notice also that, by the construction of $\gamma$, we have $m-k-j\geq 0$ and $j\geq 0$, so that $m-k\geq j\geq 0$, whence our three arcs are indeed defining arcs. Thus our triangulation contains three defining arcs $\alpha, \beta, \gamma$ and is therefore contained in the set $S_{(a_i,b_{m-k},b_j)}$. Now suppose that there is no bridging arc crossing the fdl with endpoint $b_{m-k}$. Note that we do have a bridging arc with endpoint $b_{m-k}$, so after applying a $2\pi$ Dehn twist this arc becomes a bridging arc with endpoint $b_{m-k}$ that crosses the fdl in another triangulation contained in $[t]$. Thus by our above work we can again find arcs $\alpha,\beta$, such that this triangulation obtain by a Dehn twist is contained in one of our sets.
\end{proof}

Next, in the following lemma, we show that the sets $S_{(-,-,-)}$ are disjoint.

\begin{lem} \label{disjoint}
$S_{(a_k,b_i,b_j)} = S_{(a_{k'},b_{i'},b_{j'})}$ if and only if $k=k', i= i', j= j'$. Moreover, \\$S_{(a_k,b_i,b_j)} \cap S_{a_{k'},b_{i'},b_{j'}} \neq \emptyset \Rightarrow S_{(a_k,b_i,b_j)}=S_{(a_{k'},b_{i'},b_{j'})}$.
\end{lem}
 \begin{proof}
     The if direction is clear. The only if direction will be shown in the proof of the moreover statement, whose proof we now give. Suppose that $S_{(a_k,b_i,b_j)} \cap S_{a_{k'},b_{i'},b_{j'}} \neq \emptyset$. By construction, we have the arcs $\gamma_{i,j},\gamma_{i',j'}$. These do not intersect if and only if either $i\leq i' \bmod m, j'\leq j\bmod m$ or $i'\leq i \bmod m , j\leq j' \bmod m$. Without loss of generality suppose $i\leq i' \bmod m, j'\leq j \bmod m$, and denote by $\gamma$ the arc $\gamma_{i,j}$. If $i< i'$ then $\alpha_{i'}^k$ would intersect $\gamma$, a contradiction. Thus $i=i'$. Similarly, we obtain $j=j'$ by considering the arc $\beta_{j'}^k$. Thus $\gamma_{i,j}=\gamma_{i',j'}$. Consider now the arcs $\alpha_{i}^k$, $\alpha_{i'}^{k'}=\alpha_{i}^{k'}$, as well as $\beta_{j}^k$, and $\beta_{j'}^{k
     }=\beta_{j}^{k'}$. If $k'<k$, then the arcs $\alpha_i^k$ and $\beta_j^{k'}$ would cross. If $k<k'$ the arcs $\alpha_i^{k'}$ and $\beta_j^k$ cross. Thus $k=k'$ and we can conclude that $S_{(a_k,b_i,b_j)} = S_{(a_{k'},b_{i'},b_{j'})}$.
 \end{proof}

The direct application of the three lemmas above gives Theorem \ref{finalresult}. We can think of it as a map from the set of all possible triangulations, into our constructed set. In effect, we have shown that this map is injective and surjective, and calculated the cardinality of our sets. Thus the statement of the theorem is clear. We now give proofs of the technical lemmas used in our proofs. We begin with reducing the double summation to the single summation. 

\begin{lem}\label{singlesum}
    $\sum_{j=0}^{m-1}\sum_{i=j}^{m-1}C_iC_{n+m-i-1}=\sum_{k=0}^{m-1}(k+1)C_kC_{n+m-k-1}$
\end{lem}

\begin{proof}
    This identity is evident by considering how many times the term $C_xC_{n+m-x-1}$ for $0\leq x \leq m-1$ appears in the double summation. Namely, for each $0\leq j \leq x$ it appears precisely once in the inside summation. For $j>x$ the term does not appear in the inside summation. Hence, in total it shows up $x+1$ many times. 
\end{proof}  

\begin{rem}
    It is important to note that in the above proof, by the term $C_xC_{n+m-x-1}$ we mean mean the symbolic term, not other terms that equate to the same value. For example, in the case $n=2,m=4,x=2$, when considering the term $C_2C_{2+4-2-1}$ we do not consider the term $C_3C_{2+4-3-1}$.
\end{rem}

\begin{lem}\label{Shifting Lemma}
    Let $a_i\in \mathcal{A},b_j\in \mathcal{B}$. Then up to homeomorphism,   
    $(p^{-1}(a_i)_x)\cap (p^{-1}(b_j)_x) = \emptyset$. Hence the pre-image of any bridging arc on the annulus can be viewed as a line segment with a finite positive or negative slope on the universal cover.
\end{lem}

\begin{proof}
    It is clear that the lifts of two distinct vertices in $\mathcal{A}$ ( respectively in $\mathcal{B}$) intersect trivially. Suppose we have two vertices $a_i \in \mathcal{A}$, $b_j \in \mathcal{B}$ such that there exists lifts of $a_i, b_j$ that share the same $x$-coordinate. Let $E=\big{\{} |(a'_k)_x-(b'_l)_x|\bmod 1 ; a'_k\in p^{-1}(a_k), b'_l\in p^{-1}(b_l) \forall k,l \big{\}}- \{0\}$. This set is finite due to the modular property of $p$. Intuitively, we take the pre-images of $\mathcal{A}$ and $\mathcal{B}$, put them onto a single real line, ignore the overlapping vertices, and consider the distances of the remaining points modulo 1. Put $\epsilon=\frac{1}{3}\text{min}(E)$. Then a clockwise or counter-clockwise $2\pi \epsilon$ Dehn twist of the inner boundary component gives a homeomorphic annulus such that for any $a_i\in \mathcal{A},b_j\in \mathcal{B}$,   
    $(p^{-1}(a_i)_x)\cap (p^{-1}(b_j)_x) = \emptyset$. This is because under the $2\pi \epsilon$ Dehn twist, any ``overlapped" marked points get ``un-overlapped", while not ``overlapping" any other marked points due to the choice of $\epsilon$. Formally, if $p^{-1}(a_i)_x=p^{-1}(b_j)_x$ for some $i,j$ in the annulus obtained via the $2\pi \epsilon$ Dehn twist, then on the universal cover of the original matrix, $p^{-1}(a_i)_x=p^{-1}(b_j)_x \pm \frac{\epsilon}{3}$, contradicting our choice of $\epsilon$.
\end{proof}

\section{Algebraic description of families} \label{sec: algebraic description of families}

In this section, we will provide an algebraic description of inner equivalence classes in terms of the cluster category of $Q:=Q^{\bm{\varepsilon}}$, a quiver of type $\tilde{\mathbb{A}}_{p-1}$. First, recall the definition of the cluster category $\mathcal{C}_Q=\mathscr{D}/F$ where $\mathscr{D} := \mathscr{D}^b(mod\text{-}\Bbbk Q)$ and $F:\mathscr{D} \rightarrow \mathscr{D}$ is the auto-equivalence defined by $\Sigma \tau^{-1}$. Typically, we choose the representatives of the $F$ orbits to be the unshifted objects along with the once shifted projectives in $\mathscr{D}$. In what follows, the $F$ orbit of objects and morphisms in $\mathcal{C}_Q$ will be important to provide intuition behind the algebraic consequences of performing Dehn twists of the annulus. Thus, when we draw the AR quiver of the cluster category associated to $Q^{\bm{\varepsilon}}$, we will include negatively shifted injectives and irreducible morphisms that are identified with the shifted projectives in the cluster category.

\begin{exmp}\label{exmp: cluster category}
    Let $Q$ be the quiver $$
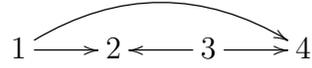
\xymatrix{1 \ar[r] \ar@/^1.5pc/[rrr] & 2 & 3 \ar[l] \ar[r] & 4}$$

    In this example, we will construct the parts of the AR quiver of $\mathcal{C}_Q$ containing only the preprojectives, preinjectives, and shifted projectives. Moreover, we will include the negatively shifted injectives and color them to indicate that they are in the same $F$ orbit as the corresponding shifted projective, which is the  typically chosen representative. The reason we are doing this will be explained later in the section. The way we have constructed this AR quiver, and the way we will always construct the AR quiver of $\mathcal{C}_Q$ is as follows. The arrows with positive slope between two string modules indicate the addition (deletion) of a hook (cohook) at the start of the string. The downward sloping arrows between two string modules indicate the addition (deletion) of a hook (cohook) at the end of the string. With this convention, the position of the shifted projectives is uniquely defined.
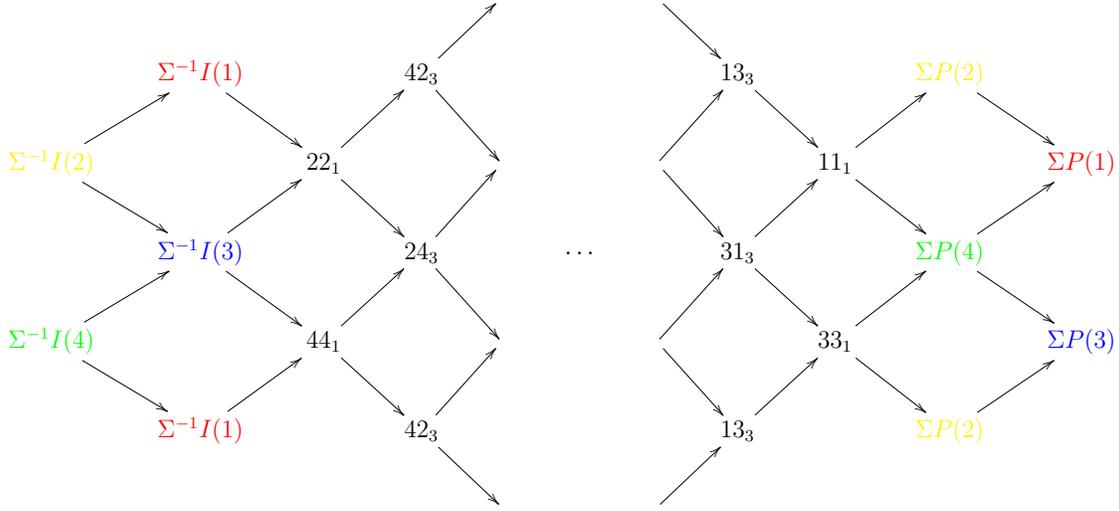
\begin{center}
\resizebox{15cm}{!}{
    $$\xymatrix{ 
     & & & & & & \ar[dr] & & & & \\ 
     & {\color{red}\Sigma^{-1} I(1)} \ar[dr] & & 42_3 \ar[dr] \ar[ur] & & & & 13_3 \ar[dr] & & {\color{yellow} \Sigma P(2)} \ar[dr] & \\
     {\color{yellow}\Sigma^{-1} I(2)} \ar[ur] \ar[dr] & & 22_1 \ar[ur] \ar[dr] & & & & \ar[ur] \ar[dr] & & 11_1 \ar[ur] \ar[dr] & & {\color{red}\Sigma P(1)} \\
     & {\color{blue} \Sigma^{-1} I(3)} \ar[ur] \ar[dr] & & 24_3 \ar[ur] \ar[dr] & & \cdots & & 31_3 \ar[ur] \ar[dr] & & {\color{green}\Sigma P(4)} \ar[ur] \ar[dr] & \\
     {\color{green}\Sigma^{-1} I(4)} \ar[ur] \ar[dr] & & 44_1 \ar[ur] \ar[dr] & & & & \ar[ur] \ar[dr] & & 33_1 \ar[ur] \ar[dr] & & {\color{blue}\Sigma P(3)} \\ 
     & {\color{red} \Sigma^{-1} I(1)} \ar[ur] & & 42_3 \ar[dr] \ar[ur] & & & & 13_3 \ar[ur] & & {\color{yellow} \Sigma P(2)} \ar[ur] & \\
     & & & & & & \ar[ur] & & &}$$}
     \end{center}
\end{exmp}

Let $M,N\in\text{Ob}(\mathcal{C}_Q)$. We say that $M$ \textbf{connects to} $N$ if there exist objects $M',N'\in \text{Ob}(\mathcal{C}_Q) \cup \{\Sigma^{-1} I(j) | j\in Q_0\}$ such that $M'$ and $N'$ are in the $F$ orbits of $M$ and $N$ respectively and there exists an arrow from $M'$ to $N'$ in the AR quiver of $\mathscr{D}$. By a \textbf{path} in $\mathcal{C}_Q$, we mean a sequence of objects $(\dots,M_1,M_2,\dots,M_k,\dots)$ in $\mathcal{C}_Q$ such that for each $i$, $M_i$ connects to $M_{i+1}$. We say there exists a \textbf{path} from $M$ to $N$ in $\mathcal{C}_Q$ if there exists a finite sequence of objects $(M = M_1,M_2,\dots,M_k = N)$ such that for each $i \in \{1,2,3,\dots, k-1\}$, $M_i$ connects to $M_{i+1}$. For example, there is a path from $33_1$ to $44_1$ in $\mathcal{C}_Q$ from Example \ref{exmp: cluster category} given by $(33_1,\Sigma P(4), \Sigma P(1), 44_1)$ where $\Sigma P(1)$ connects to $44_1$ since there is an arrow in the AR quiver of $\mathscr{D}$ from $\Sigma^{-1} I(1)$, which is in the $F$ orbit of $\Sigma P(1)$, to $44_1$. On the other hand, there is no path from $44_1$ to $33_1$. 

\begin{defn} \textbf{}
\begin{itemize}
    \item By a \textbf{ray} in $\mathcal{C}_Q$, we mean an infinite path $(\dots,M_1,M_2,\dots,M_k,\dots)$ such that for consecutive objects $M_i$ and $M_{i+1}$ in the path, the following hold:
    \begin{itemize}
        \item If $M_i,M_{i+1} \in \mathcal{R}$ are consecutive objects in a path, then the arrow between them has positive slope.
        \item If $M_i,M_{i+1} \in \mathcal{I} \cup \{\Sigma P(j) | j\in Q_0\}$ are consecutive objects in a path, the arrow between them also has positive slope.
        \item Otherwise, the arrow from $M_i$ to $M_{i+1}$ has negative slope.
    \end{itemize}  
    \item Dually, a \textbf{coray} in $\mathcal{C}_Q$, is an infinite path $(\dots,M_1,M_2,\dots,M_k,\dots)$ such that for consecutive objects $M_i$ and $M_{i+1}$ in the path, the following hold:
    \begin{itemize}
        \item If $M_i,M_{i+1} \in \mathcal{R}$ are consecutive objects in a path, then the arrow between them has negative slope.
        \item If $M_i,M_{i+1} \in \mathcal{I} \cup \{\Sigma P(j) | j\in Q_0\}$ are consecutive objects in a path, the arrow between them also has negative slope.
        \item Otherwise, the arrow from $M_i$ to $M_{i+1}$ has positive slope.
    \end{itemize}  
\end{itemize}
\end{defn}

\begin{rem}
    The reason the positive/negative convention switches between the preprojective and preinjective components is that when we identify the shifted projectives with the negatively shifted injectives, we must flip the shifted projectives. In a certain sense, we can intuitively think of the preprojective and preinjective components of the AR quiver of $\mathcal{C}_Q$ as being glued together on a M\"obius strip.
\end{rem}

\begin{exmp}
In $\mathcal{C}_Q$ from Example \ref{exmp: cluster category}, we have a coray given by $$(\dots, 11_1, \Sigma P(4), \Sigma P(3), 22_1, 42_3, \dots).$$ This can be seen by identifying $\Sigma P(3)$ with $\Sigma^{-1} I(3)$. We also have a ray given by $$(\dots, 33_1, \Sigma P(4), \Sigma P(1), 22_1, 24_3, \dots)$$ which can be seen by identifying $\Sigma P(1)$ with $\Sigma^{-1} I(1)$.
\end{exmp}

We wish to show that elementary twists of the annulus correspond to moving objects up and down rays/corays in $\mathcal{C}_Q$. Recall that $Q$ is a quiver of type $\tilde{\mathbb{A}}_{p-1}$ and let $A_{n,m}$ be the associated annulus with $n$ marked points on the outer boundary and $m$ marked points on the inner boundary, where $n+m = p$. In order to provide an algebraic explanation of the families of clusters, we first need the notion of an \textbf{elementary Dehn twist} of the inner boundary component, by which we mean a $2\pi/m$ Dehn twist of the inner boundary component. Elementary Dehn twists of the outer boundary component are defined analogously. We will call these \textbf{elementary twists}. We have the following result from the proof of Theorem 4.18 in \cite{Master'sThesisClustersandTriangulations}.

\begin{lem}\textbf{}[\cite{Master'sThesisClustersandTriangulations} Theorem 4.18] \label{lem: irreducible maps strings}
\begin{itemize}
    \item Let $M$ be preprojective and not projective. Then a clockwise (counter-clockwise) elementary Dehn twist of the inner boundary component corresponds to adding (deleting) a hook at the end of $M$.
    \item Let $M$ be preprojective and not projective. Then a clockwise (counter-clockwise) elementary Dehn twist of the outer boundary component corresponds to deleting (adding) a hook at the start of $M$.
    \item Let $M$ be preinjective and not injective. Then a clockwise (counter-clockwise) elementary Dehn twist of the inner boundary component corresponds to deleting (adding) a cohook at the start of $M$. 
    \item Let $M$ be preinjective and not injective. Then a clockwise (counter-clockwise) elementary Dehn twist of the outer boundary component corresponds to adding (deleting) a cohook at the end of $M$. 
\end{itemize} 
\end{lem}
\begin{rem}\label{rem: twists non proj non inj}
   It follows from Lemma \ref{lem: irreducible maps strings} that elementary twists correspond to irreducible morphisms between the corresponding modules. Following our conventions for constructing the AR quiver of $\mathcal{C}_Q$, we conclude in this case that elementary clockwise (counter-clockwise) twists of the inner boundary component correspond to moving one place forward (backward) in a ray in the cluster category. Analogously, elementary clockwise (counter-clockwise) twists of the outer boundary component correspond to moving one place backward (forward) in a coray in the cluster category.   
\end{rem}

\begin{rem}
    For completeness, we mention that performing elementary clockwise (counter-clockwise) twists of the inner boundary component correspond to $\tau^{-1}$ ($\tau)$ of the left regular modules. Performing elementary clockwise (counter-clockwise) twists of the outer boundary component correspond to $\tau$ ($\tau^{-1})$ of the right regular modules.
\end{rem}

We now need to understand the action of these twists on projective/injective modules and the shifted projectives. Actually, it follows from the proof of Theorem 4.18 in \cite{Master'sThesisClustersandTriangulations} that in the case we have an injective module such that a cohook can be deleted at the start of the string, an elementary clockwise twist of the inner boundary component corresponds to deleting a cohook at the start of the string, hence moving one position forward in a ray. Analogous statements can be made for the other twists and for projectives. Thus it suffices to study the case in which hooks and cohooks can't be added or deleted from the string.

\begin{lem} \textbf{} \label{lem: irreducible maps to shifts}
    \begin{enumerate}
        \item Let $I(j)$ be the injective module at vertex $j$ such that a cohook can't be deleted at the start (end). Then an elementary clockwise (counter-clockwise) twist of the inner (outer) boundary component sends $I(j)$ to $\Sigma P(j+1 \bmod \, |Q_0|)$ ($\Sigma P(j-1 \bmod \, |Q_0|)$).
        
        \item Let $P(j)$ be the projective module at vertex $j$ such that a hook can't be deleted at the the end (start). Then an elementary counter-clockwise (clockwise) twist of the inner (outer) boundary component sends $P(j)$ to $\Sigma P(j-1 \text{mod} \, |Q_0|)$ ($\Sigma P(j+1 \bmod \, |Q_0|)$).
    \end{enumerate}
\end{lem}

\begin{proof}
We will show the first part of 1. The proof of the second part of 1 is analogous and 2 is dual to 1. Let $I(j)$ be the injective module at vertex $j$ where $j$ is such that it is not possible to delete a cohook from the start of $I(j)$ and let $a(x,y)$ be the arc corresponding to $I(j)$. Let $l = a(l',l'')$ be the first (left most in the cover) steep arc that $a(x,y)$ crosses and since $I(j)$ is the injective module at $j$, $j = a(j',j'')$ is the last steep arc (right most in the cover) that $a(x,y)$ crosses. Since it is not possible to delete a cohook from the start of $I(j)$, we know that all arrows in the walk corresponding to $I(j)$ must be direct, hence in $Q_1$. Therefore, we conclude that all of the steep arcs that $a(x,y)$ crosses must be clockwise of $l$. Since the steep arcs are bridging, we conclude that they all have an endpoint in common. Suppose this end point lies on the inner boundary component. The proof when this endpoint lies on the outer boundary component is analogous. Moreover, since the steep arcs form a triangulation, we conclude $y$ must be the successor of $j''$ on the outer boundary component. Then a clockwise $2\pi/m$ Dehn twist sends $a(x,y)$ to the arc $a(x',y)$ beginning at the successor of $x$ and ending at the successor of $j''$. Again, since the steep arcs form a triangulation and it is impossible to delete a cohook from the start of $I(j)$, we conclude that $a(x',y)$ must be the steep arc $j+1 \bmod \, |Q_0|$. Thus $I(j)$ gets sent to $\Sigma P(j+1 \bmod \, |Q_0|)$.  
\end{proof}

\begin{rem} \label{rem: twists proj and inj}
    It follows from Lemma \ref{lem: irreducible maps to shifts} that elementary twists give irreducible morphisms in $\mathcal{C}_Q$ between injectives and shifted projectives. Thus in this case elementary clockwise (counter-clockwise) twists of the inner (outer) boundary component move the corresponding object one place forward in a ray (coray) in $\mathcal{C}_Q$. On the other hand, we see that elementary twists of projectives can give shifted projectives. To make this consistent with the connection between elementary twists and irreducible morphisms, notice that $\Sigma^{-1} I(j-1 \bmod \, |Q_0|)$ is in the same $F$ orbit as $\Sigma P(j-1 \bmod \, |Q_0|)$. With this in mind, we can think of an elementary counter-clockwise twist of the inner boundary component as sending $P(j)$ to $\Sigma^{-1} I(j-1 \bmod \, |Q_0|)$, which is one shift backward in a ray in $\mathscr{D}$. Thus by construction, $P(j)$ gets sent one shift back in a ray to $\Sigma P(j-1 \bmod \, |Q_0|)$.  Analogously, an elementary clockwise twist of the outer boundary component sends $P(j)$ one shift back in a coray to $\Sigma P(j+1 \bmod \, |Q_0|)$.
\end{rem}

Finally, the last case to consider is how the action of elementary twists affects the steep arcs, equivalently, the shifted projectives.

\begin{lem} \label{lem: twists of vertical strands}
    Let $\Sigma P(j)$ be a shifted projective. 
    \begin{enumerate}
        \item An elementary clockwise twist of the inner boundary component sends $\Sigma P(j)$ to either $\Sigma P(j+1 \bmod \, |Q_0|)$ or to $P(j+1 \bmod \, |Q_0|))$.
        \item An elementary counter-clockwise twist of the inner boundary component sends $\Sigma P(j)$ to either $\Sigma P(j-1 \bmod \, |Q_0|)$ or to $I(j-1 \bmod \, |Q_0|))$.
        \item An elementary clockwise twist of the outer boundary component sends $\Sigma P(j)$ to either $\Sigma P(j+1 \bmod \, |Q_0|)$ or to $I(j+1 \bmod \, |Q_0|))$.
        \item An elementary counter-clockwise twist of the outer boundary component sends $\Sigma P(j)$ to either $\Sigma P(j-1 \bmod \, |Q_0|)$ or to $P(j-1 \bmod \, |Q_0|))$.
    \end{enumerate}
\end{lem}

\begin{proof}
    We will prove 1. The proof of 2 is analogous to 1 and 3 and 4 are dual. Consider the object $\Sigma P(j)$ which corresponds to a steep arc labeled $j$. If an elementary clockwise twist of the inner boundary component sends the steep arc to another steep arc, then the new steep arc must have label $j+1 \bmod \, |Q_0|$ from our labeling convention. Moreover, the new steep arc must be counter-clockwise of the steep arc labeled $j$. We conclude that there is an arrow in $\mathcal{C}_Q$ from $\Sigma P(j)$ to $\Sigma P(j+1 \bmod \, |Q_0|)$, so this elementary twist corresponds to moving $\Sigma P(j)$ one position forward in a ray in $\mathcal{C}_Q$. 

    On the other hand, suppose that an elementary twist of the steep arc $j$ gives a non-steep bridging arc. Then this arc must cross all steep arcs that are clockwise of $j$. The first steep arc the twist of $j$ crosses is labeled $j+1 \bmod \, |Q_0|$. Suppose that the last steep arc that the twist of $j$ crosses is labeled $k = a(x,y)$. Consider the steep arc labeled $k+1\bmod |Q_0|$ and for a contradiction, suppose this arc is clockwise of $k$. Then it must be of the form $a(x,y')$ with $y'$ the successor of $y$ on the outer boundary component or $a(x',y)$ with $x$ the successor of $x'$ on the inner boundary component. In the first case, this contradicts the fact that $k$ is the last steep arc crossed, and in the second case, $a(x',y)$ would cross the steep arc labeled $j$ contradicting the fact that we began with a triangulation. We conclude that the steep arc labeled $k+1\bmod |Q_0|$ must be counter-clockwise of $k$ and therefore, there is an arrow in $Q$ pointing from $k+1$ to $k$. We conclude that the module associated to the twist of $j$ is precisely the projective $P(j+1 \bmod \, |Q_0|)$.
\end{proof}

\begin{rem}\label{rem: twists shifts}
    In light of Lemma \ref{lem: twists of vertical strands}, we see that in the case of steep arcs, elementary clockwise twists of the inner (outer) boundary component send shifted projectives one position forward in a ray (backward in a coray) in $\mathcal{C}_Q$. Analagously, elementary counter-clockwise twists of the inner (outer) boundary component send shifted projectives one position backward in a ray (forward in a coray) in $\mathcal{C}_Q$. 
\end{rem}

Combining Remarks \ref{rem: twists non proj non inj}, \ref{rem: twists proj and inj}, and \ref{rem: twists shifts}, we have the following proposition.

\begin{prop}\label{prop: twists in cluster cat}
    Let $M \in \text{Ob}(\mathcal{C}_Q)$ that is not regular and let $\alpha$ be its corresponding arc on the annulus $A_{Q}$. We have the following:

    \begin{itemize}
        \item A clockwise (counter-clockwise) elementary Dehn twist of the inner boundary component sends $M$ one position forward (backward) in a ray in $C_Q$.
        \item A clockwise (counter-clockwise) elementary Dehn twist of the outer boundary component sends $M$ one position forward (backward) in a coray in $C_Q$. \qed
    \end{itemize}
\end{prop}

We are now ready to provide a description of inner equivalence classes of clusters in terms of $\mathcal{C}_Q$. The proof of the following theorem follows from Proposition \ref{prop: twists in cluster cat}

\begin{thm} \label{thm: families in cluster category}
    Let $C$ and $C'$ be two clusters in $\mathcal{C}_Q$ and label their corresponding triangulations of the annulus $A_Q = A_{n,m}$ by $T$ and $T'$ respectively. Suppose that for $z\in\mathbb{N}$, the triangulation $T'$ is attained from $T$ by performing $z$ $2\pi$ clockwise Dehn twists of the inner boundary component of the annulus. Then the cluster $C'$ has all the same regular objects as $C$ and each non-regular object $M'$ in $C'$ is attained from the non-regular object $M$ of $C$ by moving $M$  $z\cdot m$ positions forward in any ray containing $M$. \qed
\end{thm}

\begin{exmp}
Consider the quiver $Q$ from Example \ref{exmp: cluster category} \\ \smallskip $$\xymatrix{1 \ar[r] \ar@/^1.5pc/[rrr] & 2 & 3 \ar[l] \ar[r] & 4}$$
Then a triangulation of $A_Q$ by steep arcs that gives this orientation is 
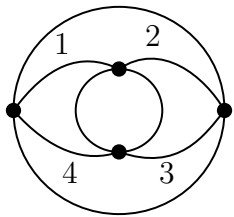
\begin{center}
\tikzset{every picture/.style={line width=0.75pt}} 
\begin{tikzpicture}[x=0.75pt,y=0.75pt,yscale=-1,xscale=1]

\draw   (303.38,83.7) .. controls (303.38,72.14) and (313.16,62.77) .. (325.22,62.77) .. controls (337.28,62.77) and (347.06,72.14) .. (347.06,83.7) .. controls (347.06,95.25) and (337.28,104.62) .. (325.22,104.62) .. controls (313.16,104.62) and (303.38,95.25) .. (303.38,83.7)(272,83.7) .. controls (272,54.81) and (295.83,31.39) .. (325.22,31.39) .. controls (354.61,31.39) and (378.44,54.81) .. (378.44,83.7) .. controls (378.44,112.58) and (354.61,136) .. (325.22,136) .. controls (295.83,136) and (272,112.58) .. (272,83.7) ;
\draw    (325.22,62.77) ;
\draw [shift={(325.22,62.77)}, rotate = 0] [color={rgb, 255:red, 0; green, 0; blue, 0 }  ][fill={rgb, 255:red, 0; green, 0; blue, 0 }  ][line width=0.75]      (0, 0) circle [x radius= 3.35, y radius= 3.35]   ;
\draw    (325.22,104.62) ;
\draw [shift={(325.22,104.62)}, rotate = 0] [color={rgb, 255:red, 0; green, 0; blue, 0 }  ][fill={rgb, 255:red, 0; green, 0; blue, 0 }  ][line width=0.75]      (0, 0) circle [x radius= 3.35, y radius= 3.35]   ;
\draw    (378.44,83.7) ;
\draw [shift={(378.44,83.7)}, rotate = 0] [color={rgb, 255:red, 0; green, 0; blue, 0 }  ][fill={rgb, 255:red, 0; green, 0; blue, 0 }  ][line width=0.75]      (0, 0) circle [x radius= 3.35, y radius= 3.35]   ;
\draw    (272,83.7) ;
\draw [shift={(272,83.7)}, rotate = 0] [color={rgb, 255:red, 0; green, 0; blue, 0 }  ][fill={rgb, 255:red, 0; green, 0; blue, 0 }  ][line width=0.75]      (0, 0) circle [x radius= 3.35, y radius= 3.35]   ;
\draw [color={rgb, 255:red, 0; green, 0; blue, 0 }  ,draw opacity=1 ]   (272,83.7) .. controls (288.44,61.39) and (305.44,53.39) .. (325.22,62.77) ;
\draw [color={rgb, 255:red, 0; green, 0; blue, 0 }  ,draw opacity=1 ]   (272,83.7) .. controls (290.44,104.39) and (308.44,110.39) .. (325.22,104.62) ;
\draw [color={rgb, 255:red, 0; green, 0; blue, 0 }  ,draw opacity=1 ]   (325.22,62.77) .. controls (346.44,48.39) and (368.44,66.39) .. (378.44,83.7) ;
\draw [color={rgb, 255:red, 0; green, 0; blue, 0 }  ,draw opacity=1 ]   (378.44,83.7) .. controls (368.44,96.18) and (355.44,115.18) .. (325.22,104.62) ;

\draw (291,43.4) node [anchor=north west][inner sep=0.75pt]    {$1$};
\draw (337,39.4) node [anchor=north west][inner sep=0.75pt]    {$2$};
\draw (344,108.4) node [anchor=north west][inner sep=0.75pt]    {$3$};
\draw (295,108.4) node [anchor=north west][inner sep=0.75pt]    {$4$};

\end{tikzpicture}
\end{center}

Consider the following triangulations that are inner equivalent.
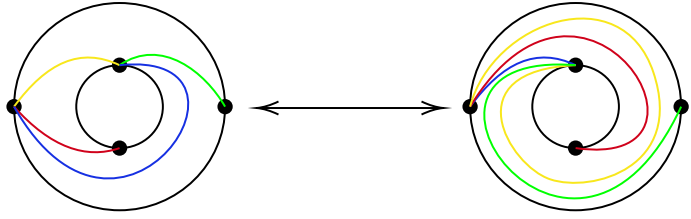
\begin{center}
\tikzset{every picture/.style={line width=0.75pt}} 
\begin{tikzpicture}[x=0.75pt,y=0.75pt,yscale=-1,xscale=1]

\draw   (184.38,88.7) .. controls (184.38,77.14) and (194.16,67.77) .. (206.22,67.77) .. controls (218.28,67.77) and (228.06,77.14) .. (228.06,88.7) .. controls (228.06,100.25) and (218.28,109.62) .. (206.22,109.62) .. controls (194.16,109.62) and (184.38,100.25) .. (184.38,88.7)(153,88.7) .. controls (153,59.81) and (176.83,36.39) .. (206.22,36.39) .. controls (235.61,36.39) and (259.44,59.81) .. (259.44,88.7) .. controls (259.44,117.58) and (235.61,141) .. (206.22,141) .. controls (176.83,141) and (153,117.58) .. (153,88.7) ;
\draw    (206.22,67.77) ;
\draw [shift={(206.22,67.77)}, rotate = 0] [color={rgb, 255:red, 0; green, 0; blue, 0 }  ][fill={rgb, 255:red, 0; green, 0; blue, 0 }  ][line width=0.75]      (0, 0) circle [x radius= 3.35, y radius= 3.35]   ;
\draw    (206.22,109.62) ;
\draw [shift={(206.22,109.62)}, rotate = 0] [color={rgb, 255:red, 0; green, 0; blue, 0 }  ][fill={rgb, 255:red, 0; green, 0; blue, 0 }  ][line width=0.75]      (0, 0) circle [x radius= 3.35, y radius= 3.35]   ;
\draw    (259.44,88.7) ;
\draw [shift={(259.44,88.7)}, rotate = 0] [color={rgb, 255:red, 0; green, 0; blue, 0 }  ][fill={rgb, 255:red, 0; green, 0; blue, 0 }  ][line width=0.75]      (0, 0) circle [x radius= 3.35, y radius= 3.35]   ;
\draw    (153,88.7) ;
\draw [shift={(153,88.7)}, rotate = 0] [color={rgb, 255:red, 0; green, 0; blue, 0 }  ][fill={rgb, 255:red, 0; green, 0; blue, 0 }  ][line width=0.75]      (0, 0) circle [x radius= 3.35, y radius= 3.35]   ;
\draw   (414.38,88.7) .. controls (414.38,77.14) and (424.16,67.77) .. (436.22,67.77) .. controls (448.28,67.77) and (458.06,77.14) .. (458.06,88.7) .. controls (458.06,100.25) and (448.28,109.62) .. (436.22,109.62) .. controls (424.16,109.62) and (414.38,100.25) .. (414.38,88.7)(383,88.7) .. controls (383,59.81) and (406.83,36.39) .. (436.22,36.39) .. controls (465.61,36.39) and (489.44,59.81) .. (489.44,88.7) .. controls (489.44,117.58) and (465.61,141) .. (436.22,141) .. controls (406.83,141) and (383,117.58) .. (383,88.7) ;
\draw    (436.22,67.77) ;
\draw [shift={(436.22,67.77)}, rotate = 0] [color={rgb, 255:red, 0; green, 0; blue, 0 }  ][fill={rgb, 255:red, 0; green, 0; blue, 0 }  ][line width=0.75]      (0, 0) circle [x radius= 3.35, y radius= 3.35]   ;
\draw    (436.22,109.62) ;
\draw [shift={(436.22,109.62)}, rotate = 0] [color={rgb, 255:red, 0; green, 0; blue, 0 }  ][fill={rgb, 255:red, 0; green, 0; blue, 0 }  ][line width=0.75]      (0, 0) circle [x radius= 3.35, y radius= 3.35]   ;
\draw    (489.44,88.7) ;
\draw [shift={(489.44,88.7)}, rotate = 0] [color={rgb, 255:red, 0; green, 0; blue, 0 }  ][fill={rgb, 255:red, 0; green, 0; blue, 0 }  ][line width=0.75]      (0, 0) circle [x radius= 3.35, y radius= 3.35]   ;
\draw    (383,88.7) ;
\draw [shift={(383,88.7)}, rotate = 0] [color={rgb, 255:red, 0; green, 0; blue, 0 }  ][fill={rgb, 255:red, 0; green, 0; blue, 0 }  ][line width=0.75]      (0, 0) circle [x radius= 3.35, y radius= 3.35]   ;
\draw    (277.44,89.39) -- (367.44,89.39) ;
\draw [shift={(369.44,89.39)}, rotate = 180] [color={rgb, 255:red, 0; green, 0; blue, 0 }  ][line width=0.75]    (10.93,-3.29) .. controls (6.95,-1.4) and (3.31,-0.3) .. (0,0) .. controls (3.31,0.3) and (6.95,1.4) .. (10.93,3.29)   ;
\draw [shift={(275.44,89.39)}, rotate = 0] [color={rgb, 255:red, 0; green, 0; blue, 0 }  ][line width=0.75]    (10.93,-3.29) .. controls (6.95,-1.4) and (3.31,-0.3) .. (0,0) .. controls (3.31,0.3) and (6.95,1.4) .. (10.93,3.29)   ;
\draw [color={rgb, 255:red, 248; green, 231; blue, 28 }  ,draw opacity=1 ]   (153,88.7) .. controls (169.44,66.39) and (186.44,58.39) .. (206.22,67.77) ;
\draw [color={rgb, 255:red, 208; green, 2; blue, 27 }  ,draw opacity=1 ]   (153,88.7) .. controls (171.44,109.39) and (189.44,115.39) .. (206.22,109.62) ;
\draw [color={rgb, 255:red, 0; green, 255; blue, 0 }  ,draw opacity=1 ]   (206.22,67.77) .. controls (227.44,53.39) and (249.44,71.39) .. (259.44,88.7) ;
\draw [color={rgb, 255:red, 22; green, 48; blue, 226 }  ,draw opacity=1 ]   (206.22,67.77) .. controls (287.44,59.39) and (207.44,184.39) .. (153,88.7) ;
\draw [color={rgb, 255:red, 22; green, 48; blue, 226 }  ,draw opacity=1 ]   (383,88.7) .. controls (399.44,66.39) and (416.44,58.39) .. (436.22,67.77) ;
\draw [color={rgb, 255:red, 208; green, 2; blue, 27 }  ,draw opacity=1 ]   (383,88.7) .. controls (441.44,-6.61) and (519.44,121.39) .. (436.22,109.62) ;
\draw [color={rgb, 255:red, 248; green, 231; blue, 28 }  ,draw opacity=1 ]   (383,88.7) .. controls (394.38,45.25) and (459.44,23.39) .. (475.44,71.39) .. controls (491.44,119.39) and (444.44,132.39) .. (423.44,125.39) .. controls (402.44,118.39) and (374.44,70.39) .. (436.22,67.77) ;
\draw [color={rgb, 255:red, 0; green, 255; blue, 0 }  ,draw opacity=1 ]   (436.22,67.77) .. controls (329.44,59.39) and (434.44,207.39) .. (489.44,88.7) ;

\end{tikzpicture}
\end{center}

The triangulation on the left corresponds to the cluster $$(I(3) = S(3)) \oplus \Sigma P(1) \oplus \Sigma P(2) \oplus \Sigma P(4)$$ and the one on the right corresponds to $$(P(2) = S(2)) \oplus (24_3 = P(3)) \oplus (P(4) = S_4) \oplus \Sigma P(1).$$ Below is the first cluster depicted in the cluster category. 
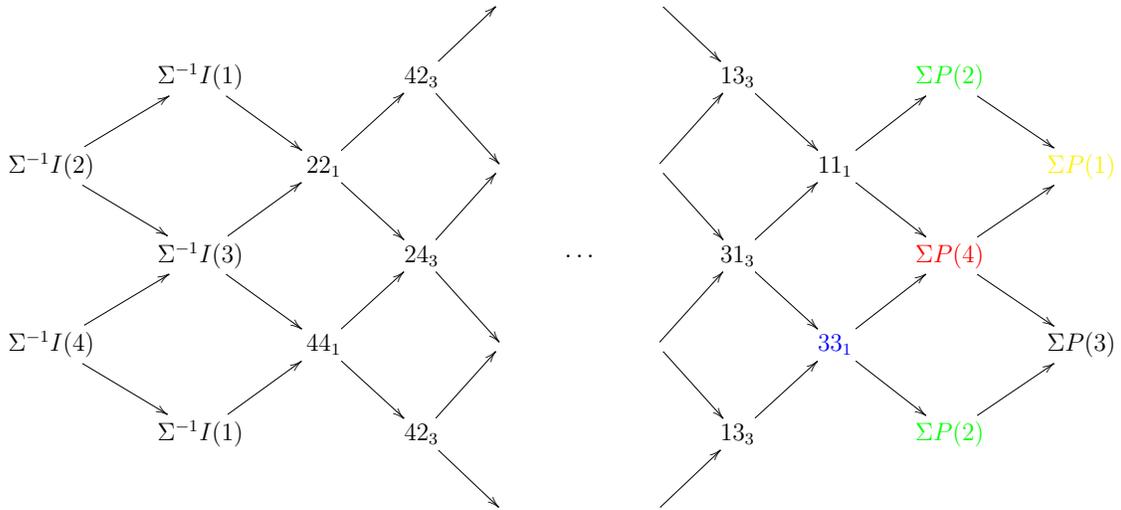
\begin{center}
\resizebox{15cm}{!}{
$$\xymatrix{ 
     & & & & & & \ar[dr] & & & & \\ 
     & \Sigma^{-1} I(1) \ar[dr] & & 42_3 \ar[dr] \ar[ur] & & & & 13_3 \ar[dr] & &  {\color{green}\Sigma P(2)} \ar[dr] & \\
     \Sigma^{-1} I(2) \ar[ur] \ar[dr] & & 22_1 \ar[ur] \ar[dr] & & & & \ar[ur] \ar[dr] & & 11_1 \ar[ur] \ar[dr] & & {\color{yellow}\Sigma P(1)} \\
     & \Sigma^{-1} I(3) \ar[ur] \ar[dr] & & 24_3 \ar[ur] \ar[dr] & & \cdots & & 31_3 \ar[ur] \ar[dr] & & {\color{red}\Sigma P(4)} \ar[ur] \ar[dr] & \\
     \Sigma^{-1} I(4) \ar[ur] \ar[dr] & & 44_1 \ar[ur] \ar[dr] & & & & \ar[ur] \ar[dr] & & {\color{blue}33_1} \ar[ur] \ar[dr] & & \Sigma P(3) \\ 
     & \Sigma^{-1} I(1) \ar[ur] & & 42_3 \ar[dr] \ar[ur] & & & & 13_3 \ar[ur] & & {\color{green}\Sigma P(2)} \ar[ur] & \\
     & & & & & & \ar[ur] & & &}$$}
\end{center}
To attain the second cluster, we move each object forward two positions in their corresponding ray as follows.

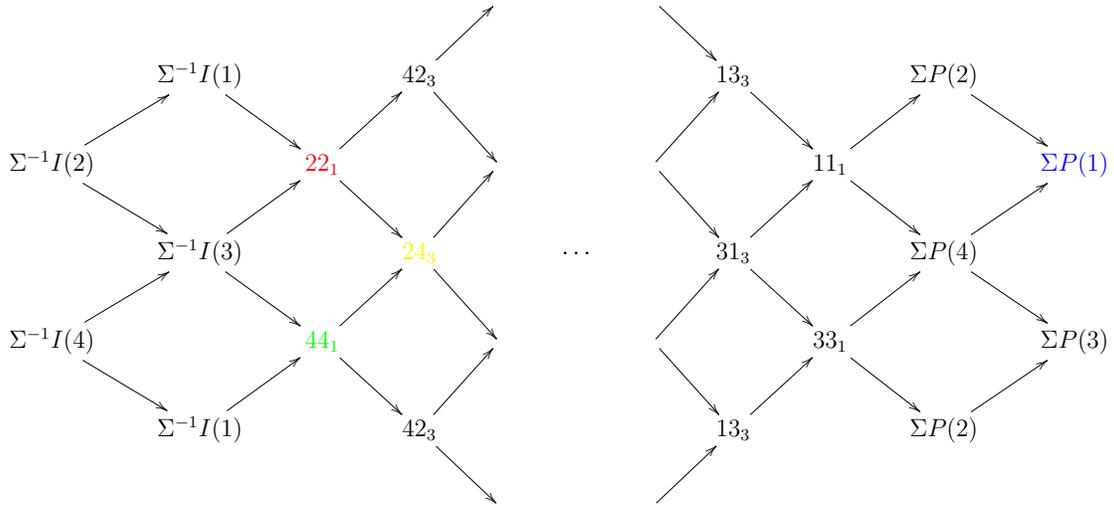
\begin{center}
\resizebox{15cm}{!}{
$$\xymatrix{ 
     & & & & & & \ar[dr] & & & & \\ 
     & \Sigma^{-1} I(1) \ar[dr] & & 42_3 \ar[dr] \ar[ur] & & & & 13_3 \ar[dr] & &  \Sigma P(2) \ar[dr] & \\
     \Sigma^{-1} I(2) \ar[ur] \ar[dr] & & {\color{red} 22_1} \ar[ur] \ar[dr] & & & & \ar[ur] \ar[dr] & & 11_1 \ar[ur] \ar[dr] & & {\color{blue}\Sigma P(1)} \\
     & \Sigma^{-1} I(3) \ar[ur] \ar[dr] & & {\color{yellow} 24_3} \ar[ur] \ar[dr] & & \cdots & & 31_3 \ar[ur] \ar[dr] & & \Sigma P(4) \ar[ur] \ar[dr] & \\
     \Sigma^{-1} I(4) \ar[ur] \ar[dr] & & {\color{green} 44_1} \ar[ur] \ar[dr] & & & & \ar[ur] \ar[dr] & & 33_1 \ar[ur] \ar[dr] & & \Sigma P(3) \\ 
     & \Sigma^{-1} I(1) \ar[ur] & & 42_3 \ar[dr] \ar[ur] & & & & 13_3 \ar[ur] & & \Sigma P(2) \ar[ur] & \\
     & & & & & & \ar[ur] & & &}$$
    }
    \end{center}
\end{exmp}

\section{Examples and table of counts of families} \label{sec: examples}

We begin this section with Example \ref{exmp: list of values}, which is a table providing a count of the number of families of triangulations of the annulus $A_{n,m}$ for a few small values of $n$ and $m$. 

\begin{exmp}\label{exmp: list of values}

\[
\begin{array}{c|cccccc}
\tikz{\node[below left, inner sep=1pt] (m) {m};%
      \node[above right,inner sep=1pt] (n) {n};%
      \draw (m.north west|-n.north west) -- (m.south east-|n.south east);}
 & 1 & 2 & 3 & 4 & 5 & 6 \\
\hline
1 & 1 & 4 & 15 & 56 & 210 & 792\\
2 & 4 & 18 & 72 & 280 & 1080 & 4158\\
3 & 15 & 72 & 300 & 1200 & 4725 & 18480\\
4 & 56 & 280 & 1200 & 4900 & 19600 & 77616\\
5 & 210 & 1080 & 4725 & 19600 & 79380 & 317520\\
6 & 792 & 4158 & 18480 & 77616 & 317520 & 1280664\\
\end{array}\]

\end{exmp}

Now in Example \ref{exmp: counting method example}, we have a depiction of all the sets $S_{(-,-,-)}$ for $A_{2,3}$.

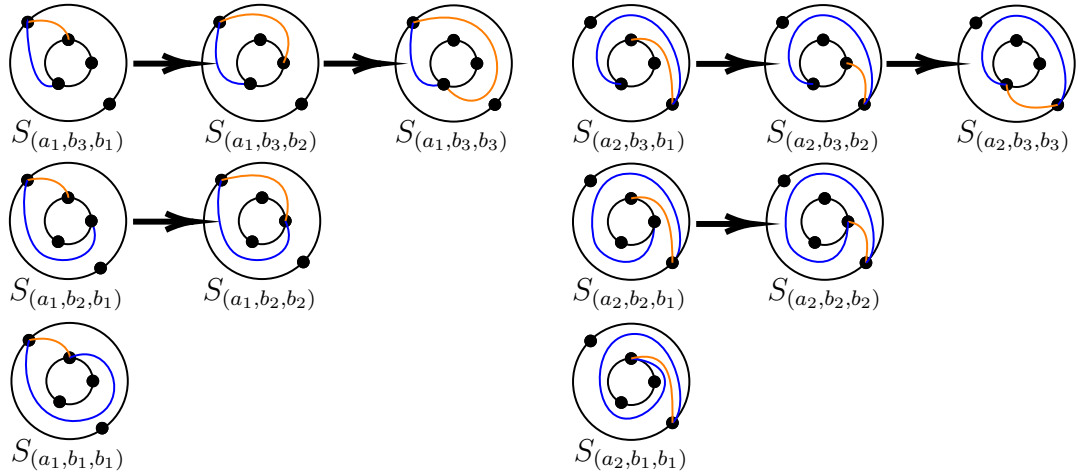
\begin{exmp}\textbf{}\\\label{exmp: counting method example}
\begin{center}
    \tikzset{every picture/.style={line width=0.75pt}} 

%
\end{center}
\end{exmp}

Finally, we end this section with Example \ref{exmp: all inner equivalence classes} which gives a complete list of representatives of inner equivalence classes of triangulations for $A_{2,2}$.

\begin{exmp}\textbf{}\label{exmp: all inner equivalence classes}
    
\begin{center}
\tikzset{every picture/.style={line width=0.75pt}} 

\tikzset{every picture/.style={line width=0.75pt}} 

\tikzset{every picture/.style={line width=0.75pt}} 



\end{center}

\end{exmp}

\eject

\nocite{*}
\bibliographystyle{amsplain}
\bibliography{bibliography}

\end{document}